\documentclass[11pt]{article}
\usepackage[utf8]{inputenc}
\usepackage{amssymb,amsfonts,amsthm}
\usepackage{amsthm,mathtools}
\usepackage{amssymb}
\usepackage[francais,english]{babel}
\usepackage{setspace}
\usepackage{bm}
\usepackage{imakeidx}
\usepackage{makeidx}
\usepackage{verbatim}
\usepackage[noblocks]{authblk}
\usepackage[dvipsnames]{color}

\usepackage{amsfonts,amsmath,layout}

\newcommand{\Z}{\mathbb Z}

\newcommand{\R}{\mathbb R}

\newcommand{\T}{\mathbb{T}}

\def\R{\mathbb R}

\def\Z{\mathbb Z}

\def\ep{\epsilon}

\def\dk{{\bf d}_1}

\newcommand{\be}{\begin{equation}}
\newcommand{\ee}{\end{equation}}
\newcommand{\ba}{\begin{align}}
\newcommand{\ea}{\end{align}}

\def\1{{\bf 1}}

\def\inte{\int_{\T^d}}
\def\dive{{\rm div}}

\newcommand*\Laplace{\mathop{}\!\mathbin\bigtriangleup}

\newtheorem{teo}{Theorem}[section]
\newtheorem{defin}[teo]{Definition}
\newtheorem{prop}[teo]{Proposition}
\newtheorem{lemma}[teo]{Lemma}

\newtheorem{remark}[teo]{Remark}

\mathtoolsset{showonlyrefs}

\usepackage[dvipsnames]{xcolor}
\definecolor{mygray}{gray}{0.7}
\usepackage[pdftex]{graphicx}
\definecolor{Red}{cmyk}{0,1,1,0.2}

\usepackage[colorinlistoftodos,prependcaption,textsize=normal]{todonotes}

\begin{document} 

\title{On the long time convergence of potential MFG}
\author{Marco Masoero \thanks{Universit\'e Paris-Dauphine, CEREMADE, Place de Lattre de Tassigny, F- 75016 Paris, France}}

\maketitle

\begin{abstract}

   We look at the long time behavior of potential Mean Field Games (briefly MFG) using some standard tools from weak KAM theory. We first show that the time-dependent minimization problem converges to an ergodic constant $-\lambda$, then we provide a class of examples where the value of the stationary MFG minimization problem is strictly greater than $-\lambda$. This will imply that the trajectories of the time-dependent MFG system do  not converge to static equilibria.
   
 \end{abstract}
 \section*{Introduction}
 Mean Field Games were first introduced by Lasry and Lions \cite{lasry2006jeux, lasry2006jeux2} and, simultaneously, by Huang, Caines and Malamh\'e \cite{huang2006large}. This theory is a branch of the broader theory of Dynamic Games and it is devoted to those models where infinitely many players interact strategically with each other. 
 
 In many cases the Nash equilibria of those games can be analyzed through the solutions of the, so called, MFG system
 $$
\begin{cases} 
-\partial_t u-\Laplace u+H(x,Du)=F(x,m) & \mbox{in } \R^d\times[0,T]\\
-\partial_t m+\Laplace m+{\rm div}(mD_pH(x,Du))=0 & \mbox{in }\R^d\times[0,T]\\
m(t)=m_0,\; u(T,x)=u_{T}(x)&\mbox{in }\R^{d}.
\end{cases}
$$
 
 with unknown the couple $(u,m)$. We can think at $m(t)$ as the distribution of players at time $t$ and $u(t,x)$ as the value function of any infinitesimal player starting from $x$ at time $t$. 
 
 The aim of this paper is to shed some light on the long time behavior of potential MFG when monotonicity is not in place. The long time behavior and the existence of solutions which are periodic in time have been subject of several papers starting from \cite{lionsmean} and the Mexican wave model in \cite{gueant2011mean} to more recent results in \cite{cardaliaguet2013long,cardaliaguet2013long2,cardaliaguet2012long,gomes2010discrete}, but in these papers either monotonicity was assumed or the MFG was not of potential type. Cirant and Nurbekyan were the first to recently provide some results in the direction of periodic solutions for non monotone MFG. Cirant in \cite{cirant2017existence} suggested the existence of non monotone configurations under which oscillatory behaviors were to be expected. Afterwards in \cite{cirant2018variational}, with Nurbekyan, they proved, through bifurcation methods,  the existence of a path of branches which corresponds to a periodic trajectory. The main difference with our work is the choice of the class of solutions. In our case we look at paths which are energy minimizers whereas, in their, it might not be the case. 
 
 Potential MFG are those games whose MFG system can be derived as optimality condition of the following minimization problem
 $$
\mathcal U(T,m_{0})= \inf_{(m,w)}\int_0^T\int_{\T^d}H^*\left(x,-\frac{dw(t)}{dm(t)}(x)\right)dm(t)+\mathcal F(m(t))dt,
 $$
 where $(m,w)$ verifies the Fokker Plank equation $-\partial_{t}m+\Laplace m+{\rm div}w=0$ with $m(0)=m_{0}$ and the coupling function $F$ in the MFG system is the derivative with respect to the measure of $\mathcal F$. These games have been largely studied (see Lasry and Lions \cite{lasry2006jeux2} for existence results and, among others \cite{cardaliaguet2015second,briani2016stable,meszaros2018variational} for further properties) but, so far, not much is known regarding their long time behavior outside the assumption of monotonicity, where Cardaliaguet, Lasry, Lions and Porretta \cite{cardaliaguet2013long2} proved the convergence to the ergodic system
 $$
 \begin{cases}
- \bar\lambda-\Laplace u+H(x,Du)=F(x,m)&\mbox{in } \T^d\\
\Laplace m+{\rm div}(mD_pH(x,Du))=0 & \mbox{in }\T^d.
 \end{cases}
 $$  
 We show that in general this is not the case, even in the very regular setting of non local coupling. We look at the problem from the point of view of weak KAM theory. The link between the two theories is not new and it was already proposed by Cardaliaguet \cite{cardaliaguet2013long} in the first order monotone case, even though in a different manner. 
 
 The paper is divided in three sections. In the first one we prove the convergence of $T^{-1}\mathcal U(T,\cdot)$ when $T$ goes to infinity. The method we use is directly inspired by Lions, Papanicolaou, and Varadhan \cite{lions1987homogenization}: instead of looking directly at $\lim_{T}T^{-1}\mathcal U$ we define the infinite horizon, discounted problem 
 
 $$
 \mathcal V_{\delta}(m_{0})=\inf_{(m,w)}\int_0^\infty e^{-\delta t}\int_{\T^d}H^*\left(x,-\frac{w(t)}{m(t)}\right)dm(t)+\mathcal F(m(t))dt
 $$
 
 and we prove that $\lim_{\delta}\delta\mathcal V_{\delta}(\cdot)=-\lambda$ when $\delta\rightarrow0^{+}$and that this limit is uniform with respect to the initial distribution. A key assumption is the boundedness of the second derivative of $F(x,m)$ with respect to the state variable. This gives uniform semiconcavity estimates of the solutions of the MFG system associated to the discounted minimization problem. The existence of the limit $\lim_{\delta}\delta\mathcal V_{\delta}(\cdot)$ implies the existence of the limit $\lim_{T}T^{-1}\mathcal U$ and the two must coincide. 
 
 As byproduct, we have the existence of a corrector function $\chi$ on the space of measures which enjoys the following dynamic programming principle
 $$
 \chi(m_{0})= \inf_{(m,w)}\left( \int_{t_{1}}^{t_{2}}\int_{\T^{d}}H^{*}\left(x,-\frac{w(t)}{m(t)}\right)dm(s)+\mathcal F(m(s))ds+\chi (m(t_{2}))\right)+\lambda (t_{2}-t_{1}).
 $$

 The second section is devoted to the study of the set of corrector functions. A corrector is any continuous function on the space of measures which verifies the dynamic programming principle above. Both the terminology and the techniques are borrowed from weak KAM theory, in particular we rely on Fathi's book \cite{fathi2008weak}, along with his seminal papers \cite{fathi1997solutions,fathi1997theoreme,fathi1998convergence}.  In principle, as in the standard weak KAM theory, the corrector functions verify an HJB equation in the space of probability measure. In this work nothing is said about this property which is the subject of a paper that is still in progress.
 
 Particular interest is given to the projected Mather set which is the set of probability measures contained in a calibrated curve. We say that $(\bar m,\bar w)$ is a calibrated curve associated to a corrector function $\chi$ if, for any $t_{1}$, $t_{2}\in\R$, $(\bar m,\bar w)$ is optimal for the dynamic programming principle, that is
 $$
 \chi(\bar m(t_{1}))=  \int_{t_{1}}^{t_{2}}\int_{\T^{d}}H^{*}\left(x,-\frac{\bar w(t)}{\bar m(t)}\right)d\bar m(s)+\mathcal F(\bar m(s))ds+\chi (\bar m(t_{2}))+\lambda (t_{2}-t_{1}).
 $$
 These curves play a fundamental role to understand the long time behavior of these MFG. They are indeed the attractors of the dynamics which minimize the discounted, infinite horizon MFG. 
  
 In the third section we focus on the relation between the limit value $\lambda$ and the ergodic value $\bar\lambda$, associated to the stationary MFG, defined by 
 $$
-\bar\lambda=\inf_{(m,w)}\int_{\T^d}H^*\left(x,-\frac{w}{m}\right)dm+\mathcal F(m).
 $$
 We propose two examples which highlight how much important it is the structure of the coupling function $F(x,\cdot)$ in the dynamic of potential MFG. In the first example we impose monotonicity and we recover part of the results in \cite{cardaliaguet2013long2}. In this case the limit value and the ergodic one coincide. 
 
 On the other hand, in the second example, the minimization problems are no longer convex and we can prove that $\lambda>\bar\lambda$. This means that it is not possible that the MFG system converges to a stationary equilibrium. The fact that  $\lambda>\bar\lambda$ implies that the energy of the finite horizon game goes below the energy of the stationary one. Looking at the projected Mather set we can say even more. As this set is compact and it can not contain any stationary curve, calibrated curves can not even approach any static configuration.
 
 \subsection*{Acknowledgement}
I would like to thank Pierre Cardaliaguet (Paris Dauphine) for the fruitful discussions all along this work and Marco Cirant (Universit\`a di Padova) for a crucial hint which resulted in Lemma \ref{exdualpl}.

I was partially supported by the ANR (Agence Nationale de la Recherche) project ANR-16-CE40-0015-01.

 \subsection*{Assumptions and definitions}
 We work on the $d-$dimensional flat torus $\T^{d}=\R^{d}/\Z^{d}$ to avoid to deal with boundary conditions and to set the problem on a compact domain. 
 
 \textbf{Notation:} We denote by $\mathcal P(\T^{d})$ the set of Borel probability measures on $\T^{d}$. This is a compact, complete and separable set when endowed with the Monge-Kantorovich distance $\bm d(\cdot,\cdot)$. We define $\mathcal M(\T^{d};\R^{d})$ the set of Borel vector valued  measures $w$ with finite mass $|w|$. If $m_{t}$ is a time dependent probability measure on $\T^{d}$, then $L^{2}_{m}([0,T]\times\T^{d})$ is the set of $m$-measurable functions $f$ such that the integral of $|f|^{2}dm_{t}$ over $[0,T]\times\T^{d}$ is finite. Analogously for $L^{2}_{m}(\T^{d})$ and $L^{2}_{m}(\T^{d};\R^{d})$, where in the latter case we consider vector valued functions. 
 
 We use throughout the paper the notion of derivative for functions defined on $\mathcal P(\T^{d})$ introduced in \cite{cardaliaguet2015master}. We say that $\Phi:\mathcal P(\T^{d})\rightarrow \R$ is $C^{1}$ if there exists a continuous function $\frac{\delta \Phi}{\delta m}:\mathcal P(\T^{d})\times\T^{d}\rightarrow\R$ such that
 $$
\Phi(m_{1})-\Phi(m_{2})=\int_{0}^{1}\int_{\T^{d}} \frac{\delta \Phi}{\delta m}((1-t)m_{1}+tm_{2},x)(m_{2}-m_{1})(dx)dt,\qquad\forall m_{1},m_{2}\in\mathcal P(\T^{d}).
 $$
 As this derivative is defined up to an additive constant, we use the standard normalization
 $$
\int_{\T^{d}} \frac{\delta \Phi}{\delta m}(m,x)m(dx)=0.
 $$
 
 \textbf{Assumptions:} We impose the following assumptions on the hamiltonian $H$ and the coupling function $F$ so that we can derive uniform estimates on the solutions of the MFG system.
\begin{enumerate}
\item $H:\T^d\times\R^d\rightarrow\R$ is of class $C^2$, $p\mapsto D_{pp}H(x,p)$ is Lipschitz continuous, uniformly with respect to $x$. Moreover there exists $\bar C>0$ that verifies 
\be\label{1}
\bar C^{-1}I_d\leq D_{pp}H(x,p)\leq \bar CI_d, \quad\forall (x,p)\in\T^d\times\R^d
\ee
and $\theta\in(0,1)$, $C>0$ such that the following conditions hold true
\be\label{2}
|D_{xx}H(x,p)|\leq C(1+|p|)^{1+\theta},\quad|D_{x,p}H(x,p)|\leq C(1+|p|)^\theta,\quad\forall (x,p)\in\T^d\times\R^d.
\ee

\item $\mathcal F:\mathcal P(\T^{d})\rightarrow\R$ is of class $C^2$. Its derivative $F:\T^d\times\mathcal P(\T^d)\rightarrow\R$ is twice differentiable in $x$ and $D^2_{xx}F$ is bounded.
Examples of non monotone coupling functions which verify such conditions can be found in Lemma \ref{199} and Lemma \ref{200}. 

\end{enumerate}
We recall that, if $\mu$, $\nu\in\mathcal P(\T^d)$, the $1$-Wasserstein distance is defined by
\be\label{RKd2}
\bm d(\mu,\nu)= \sup \left\{ \left. \int_{\T^{d}} \phi(x) \,d (\mu - \nu) (x) \right| \mbox{continuous } \phi : \T^{d} \to \mathbb{R},\, \mathrm{Lip} (\phi) \leq 1 \right\}.
\ee

\textbf{Minimization Problems:} Under the above assumptions, we can introduce two minimization problems. Each one of those will be proposed in two different but equivalent forms. The first one is 
\be\label{VlI}
\mathcal U(T,m_0)=\inf_{(m,\alpha)}\int_0^T \int_{\T^d}H^*(x,\alpha)dm(t)+\mathcal F(m(t))dt,\quad m_0\in\mathcal P(\T^d)
\ee
where $m\in C^0([0,T],\mathcal P(\T^d))$, $\alpha\in L^2_{m}([0,T]\times\T^d,\R^d)$ and the following equation is verified in sense of distribution
\begin{equation}\label{fp}
\begin{cases}
-\partial_t m +\Laplace m+{\rm div}(m\alpha)=0&\mbox{in }[0,T]\times\T^{d}\\
m(0)=m_0&\mbox{in }\T^{d}.
\end{cases}
\end{equation}
Equivalently (see \cite{briani2016stable} for more details),

$$
\mathcal U(T,m_0)=\inf_{(m,w)\in\mathcal E^{T}_{2}(m_{0})}\int_0^T\int_{\T^d}H^*\left(x,-\frac{dw(t)}{dm(t)}(x)\right)dm(t)+\mathcal F(m(t))dt,\quad m_0\in\mathcal P(\T^d)
$$

where $\mathcal E^{T}_{2}(m_{0})$ is the set of time dependent Borel measures $(m(t),w(t))\in\mathcal P(\T^{d})\times\mathcal M(\T^{d};\R^{d})$ such that $m\in C^0([0,T],\mathcal P(\T^d))$, $w$ is absolutely continuous with respect to $m$, its density $dw/dm$ belongs to $L^{2}_{m}([0,T]\times\T^{d})$ and $-\partial_{t}m+\Laplace m-{\rm div}w=0$ is verified in sense of distributions with initial condition $m(0)=m_{0}$.

The second minimization problem reads
\be\label{Vl}
\mathcal V_\delta(m_0)=\inf_{(m,\alpha)}\int_0^\infty e^{-\delta t}\int_{\T^d}H^*(x,\alpha)dm(t)+\mathcal F(m(t))dt,\quad m_0\in\mathcal P(\T^d)
\ee
 where $\delta>0$,  $m\in C^0([0,+\infty),\mathcal P(\T^d))$, $\alpha\in L^2_{m,\delta}([0,+\infty)\times\T^d,\R^d)$, that is $L_{m}^{2}$ with weight $e^{-\delta t}$, and $(m,\alpha)$ verifies \eqref{fp} in $[0,+\infty)\times\T^{d}$.
 Equivalently, $\mathcal V_{\delta}$ can be defined as
 \be\label{discprob}
 \mathcal V_{\delta}(m_0)=\inf_{(m,w)\in\mathcal E^{\delta}_{2}(m_{0})}\int_0^{+\infty}e^{-\delta t} \int_{\T^d}H^*\left(x,-\frac{dw(t)}{dm(t)}(x)\right)dm(t)+\mathcal F(m(t))dt,\quad m_0\in\mathcal P(\T^d)
 \ee
 where $\mathcal E^{\delta}_{2}(m_{0})$ is defined as $\mathcal E_{2}(m_{0})$ with the only difference that we ask $dw/dm$ to be $L^{2}$ in $[0,+\infty)\times\T^{d}$ with respect to $e^{-\delta t}m(t)$.
 For convenience we introduce the functional on $\mathcal E_{2}(m_{0})$
 \be\label{Jfunct}
 J_{\delta}(m_{0},m,w)=\int_0^{+\infty}e^{-\delta t} \int_{\T^d}H^*\left(x,-\frac{dw(t)}{dm(t)}(x)\right)dm(t)+\mathcal F(m(t))dt,
 \ee
so that $\mathcal V_{\delta}(m_{0})=\inf_{(m,w)} J(m_{0},m,w)$.
 
We also define the ergodic value $\bar\lambda$ as follows
\be\label{elambda}
-\bar\lambda=\inf_{m,\alpha}\int_{\T^d}H^*(x,\alpha)dm+\mathcal F(m)
\ee
where $(m,\alpha)\in\mathcal P(\T^{d})\times L^{2}_{m}(\T^{d};\R^{d})$ verifies in sense of distribution $\Laplace m+{\rm div}(m\alpha)=0$ in $\T^{d}$.

Or, equivalently,
\be\label{staticfunc3}
-\bar\lambda=\inf_{(m,w)\in\mathcal E}\int_{\T^d}H^*\left(x,-\frac{dw}{dm}(x)\right)dm(x)+\mathcal F(m)
\ee
where $\mathcal E$ is the set of $(m,w)\in\mathcal P(\T^{d})\times\mathcal M(\T^{d};\R^{d})$ such that $w$ is absolutely continuous with respect to $m$, its density $dw/dm$ belongs to $L^{2}_{m}(\T^{d})$ and $\Laplace m-{\rm div}w=0$ is verified in sense of distributions.

Throughout the paper we will use the constant $C>0$ which may change from line to line.

\section{Ergodic limit value}

\subsection{Minimizers and dynamic programming principle for $\mathcal V_{\delta}$}
We start proving that the minimization problem \eqref{discprob} admits a minimizer and we also give a characterization of such a minimizer in terms of solutions of the associated MFG system.

\begin{prop}\label{exmin} For any $\delta> 0$ and any $m_{0}\in\mathcal P(\T^{d})$, $\mathcal V_{\delta}(m_{0})$ admits a minimizer $(m,w)$. Moreover there exists $u\in C^{1,2}([0,+\infty)\times\T^d) $ and $m\in C^0([0,+\infty)\times\mathcal P(\T^d))$ solutions of
\begin{equation}\label{psystem}
\begin{cases} 
-\partial_t u-\Laplace u+\delta u+H(x,Du)=\frac{\delta \mathcal F}{\delta m}(x,m):=F(x,m) & \mbox{in } \T^d\times[0,+\infty)\\
-\partial_t m+\Laplace m+{\rm div}(mD_pH(x,Du))=0 & \mbox{in }\T^d\times[0,+\infty)\\
m(0)=m_0,\; u\in L^{\infty}([0,+\infty)\times\T^{d})
\end{cases}
\end{equation}
such that $w=-mD_{p}H(x,Du)$.
\end{prop}

\begin{proof} First of all we show that $\mathcal V_{\delta}(m_{0})$ is finite and that it is bounded by a constant $K_{\delta}$ independent of $m_{0}$. We can always use as competitor for $\mathcal V_{\delta}(m_{0})$ the couple $(\mu,0)$ where $\mu$ is the solution of 
\be\label{primo}
\begin{cases}
-\partial_{t}\mu+\Laplace\mu=0& \mbox{in } \T^d\times[0,+\infty)\\
\mu(0)=m_{0}& \mbox{in } \T^d.
\end{cases}
\ee
Given that $\mathcal F$ is bounded, if we use $(\mu,0)$ as a competitor, we get
$$
\mathcal V_{\delta}(m_{0})\leq\int_{0}^{+\infty}e^{-\delta t} H^{*}(x,0)+\sup_{m}\mathcal F(m)dt:=K_{\delta}.
$$ 

We fix a minimizing sequence $(m_{n},w_{n})$. If we use \eqref{1}, we have that
 
 $$\int_{0}^{+\infty}\int_{\T^{d}}e^{-\delta t}|w_{n}(t,x)|dxdt\leq M_{\delta}$$
 
  where $M_{\delta}$ is a fixed constant that does not depend on $m_{0}$ and $n$. Hence, for any fixed $k>0$ 
\be\label{wnbound}
\int_{0}^{k}\int_{\T^{d}}|w_{n}(t,x)|dxdt\leq M_{\delta}e^{\delta k}.
\ee

 Following Lemma 3.1 in \cite{cardaliaguet2015second}, we get that for any $t,s\in[0,k]$
\be\label{unifmn}
\bm d(m_{n}(t),m_{n}(s))\leq C^{k}_{\delta}|t-s|+C|t-s|^{1/2},
\ee
where $C^{k}_{\delta}$ depends only on $\delta$ and $k$. 
Inequality \eqref{unifmn} tells us that $\{m_{n}\}_{n}$ is uniformly bounded in $C^{1/2}([0,k),\mathcal P(\T^{d}))$. We have then that $m_{n}$ converges uniformly on any compact set to a limit $\bar m\in C^{0}([0,+\infty),\mathcal P(\T^{d}))$. Thanks to the bounds \eqref{wnbound} we also know that $w_{n}$ converges in $\mathcal M (I\times\T^{d};\R^{d})$ to a certain $\bar w$ on any bounded interval $I\subset [0,+\infty)$. 

As the couple $(\bar m,\bar w)$ belongs to $\mathcal E_{2}^{\delta}(m_{0})$ we have that $J_{\delta}(m_{0},\bar m,\bar w)<+\infty$. This means that

$$
J_{\delta}(m(0),\bar m,\bar w)=\lim_{k\rightarrow+\infty}\int_0^{k}e^{-\delta t} \int_{\T^d}H^*\left(x,-\frac{d\bar w(t)}{d\bar m(t)}(x)\right)d\bar m(t)+\mathcal F(\bar m(t))dt
$$


Note also that the functional is bounded from below, so there exists a constant $C_{\delta}$ such that, for any $(m,w)\in\mathcal E_{2}^{\delta}$,

$$
\int_k^{+\infty}e^{-\delta t} \int_{\T^d}H^*\left(x,-\frac{dw(t)}{dm(t)}(x)\right)dm(t)+\mathcal F(m(t))dt
$$
$$
\geq e^{-\delta k}\int_{0}^{+\infty}e^{-\delta t}\inf_{(x,p)\in\T^{d}\times\R^{d}} H^{*}(x,p)+\inf_{m}\mathcal F(m)dt=e^{-\delta k} C_{\delta}.
$$

Therefore,
$$
J_{\delta}(m_{0},m_{n},w_{n})\geq\int_0^{k}e^{-\delta t} \int_{\T^d}H^*\left(x,-\frac{d w_{n}(t)}{d m_{n}(t)}(x)\right)d m_{n}(t)+\mathcal F(m_{n}(t))dt+e^{-\delta k} C_{\delta}.
$$

Thanks to the convergence of $(m_{n},w_{n})$ on compact sets, we can pass to the limit in $n$ and we get
$$
\mathcal V_{\delta}(m_{0})\geq \int_0^{k}e^{-\delta t} \int_{\T^d}H^*\left(x,-\frac{d\bar w(t)}{d\bar m(t)}(x)\right)d\bar m(t)+\mathcal F(\bar m(t))dt+ C_{\delta}e^{-\delta k}.
$$

Taking the limit on $k$ we finally get that $\mathcal V_{\delta}(m_{0})\geq J(m_{0},\bar m,\bar w)$. 

The proof of the second statement relies again on classic tools for potential MFG (see for instance \cite{cardaliaguet2015second} or \cite{briani2016stable}). Using the convexity of $H^{*}$ and the regularity of $\mathcal F$, we can easily show that if $(\bar m,\bar w)$ is a minimizer for $\mathcal V_{\delta}(m_{0})$ then it must be a minimizer for $\bar J_{\delta}:\mathcal E_{2}^{\delta}(m_{0})\rightarrow \R$, which reads
$$
\bar J_{\delta}(m,w)= \int_0^{+\infty}e^{-\delta t} \int_{\T^d}H^*\left(x,-\frac{dw(t)}{dm(t)}(x)\right)dm(t)+F(x,\bar m(t))dm(t)dt.
$$

As $\bar J_{\delta}$ is convex we can define its dual problem (in the sense of Fenchel-Rockafellar)
$$
\inf_{u\in C_{b}^{2}([0,+\infty)\times\T^{d})}\left\{ -\int_{\T^{d}} u(0,x)dm_{0}(x)\mbox{ where}-\partial_{t}u-\Laplace u+\delta u+H(x,Du)\leq F(x,\bar m)\right\}.
$$

Thanks to the comparison principle, the minimizer $\bar u$ is the unique solution of $-\partial_{t}\bar u-\Laplace\bar u+\delta\bar u+H(x,D\bar u)=F(x,\bar m)$ in $C^{2}_{b}([0,+\infty)\times\T^{d})$.

If we sum the two problems we get
$$
\int_0^{+\infty}e^{-\delta t} \int_{\T^d}H^*\left(x,-\frac{d\bar w(t)}{d\bar m(t)}(x)\right)d\bar m(t)+F(x,\bar m(t))d\bar m(t)dt-\int_{\T^{d}} \phi(0,x)dm_{0}(x)=0
$$

Using the differential constraints of $(\bar m,\bar w)$ and $\bar u$ we get that $\bar w=-\bar m D_{p}H(x,D\bar u)$ $\bar m$-a.e..  This means that $\bar w$ is bounded that in turn implies that $\bar m>0$ so that $\bar w=-\bar m D_{p}H(x,D\bar u)$ is verified everywhere.

\end{proof}

We now state without proof the dynamic programming principle for $\mathcal V_{\delta}$. The proof relies on standard arguments in optimal control theory (see for instance \cite{cannarsa2004semiconcave}).

\begin{lemma}\label{Vddynamic}The function $\mathcal V_{\delta}$ verifies the dynamic programming principle, which reads
$$
\mathcal V_\delta(m_0)=\inf_{(m,\alpha)\in\mathcal E_{2}^{\delta}(m_{0})}\left(\int_0^t e^{-\delta s}\int_{\T^d}H^*(x,\alpha)dm(s)+\mathcal F(m(s))ds+e^{-\delta t}\mathcal V_\delta( m(t))\right).
$$
\end{lemma}

\subsection{Existence of a corrector} 
The main result of this section is Theorem \ref{Ucont} where show that the function $\mathcal V_{\delta}(\cdot)$ is uniformly Lipschitz with respect to $\delta$. As a consequence, we have Proposition \ref{unifconv} which claims on one side that the limit $\lim_{\delta\rightarrow0}\delta\mathcal V_{\delta}(m_{0})$ is well defined and it is uniform in $m_{0}$ and, on the other, that, up to subsequence, also $\mathcal V_{\delta}(\cdot)-\mathcal V_{\delta}(m_{0})$ converges to a continuous function $\chi$. 
In Lemma \ref{lem.1.7} we prove that $\chi$ enjoys the dynamic programming principle and, therefore, we have the existence of a corrector function.

The idea behind the proof of Theorem \ref{Ucont} is the following: we want to prove that there exists a constant $\bar K>0$ independent of $\delta$, such that
$$
|\mathcal V_\delta(m_2^0)-\mathcal V_\delta(m_1^0)|\leq \bar K\bm d(m_1^0,m_2^0).
$$
We fix an horizon $T>0$, to be chosen later, and we take $(m_1(\cdot),\alpha_1(\cdot))$ a minimizer for $\mathcal V_\delta(m_1^0)$. We consider any couple $(m_2,\alpha_2)$ such that \eqref{fp} is verified in $[0,T]$, $m_2(0)=m_2^0$, $m_2(T)=m_1(T)$ and $m_2\equiv m_1$, $\alpha_2\equiv \alpha_1$ in $[T,\infty)$. The couple $(m_2,\alpha_2)$ is admissible and 
\be\label{dpp}
\mathcal V_\delta(m_2^0)\leq  \int_0^T e^{-\delta s}\int_{\T^d}H^*(x,\alpha_2)dm_2(s)+\mathcal F(m_2(s))ds+e^{-\delta T}V_\delta(m_1(T)).
\ee
Therefore,
$$
\mathcal V_\delta(m_2^0)-\mathcal V_\delta(m_1^0)\leq \int_0^T e^{-\delta t}\int_{\T^d}H^*(x,\alpha_2)dm_2(t)-H^*(x,\alpha_1)dm_1(t)+\mathcal F(m_2(t))-\mathcal F(m_1(t))dt.
$$

In order to prove the continuity of $\mathcal V_\delta$ with respect to the initial data we need to introduce some standard estimates on the solutions of the MFG system \eqref{psystem}.

\begin{lemma}\label{2est} There exists $C>0$ independent of $m_0,T,\delta$ such that, if $(u,m)$ is a classical solution of \eqref{psystem}, then

\begin{itemize}
\item $\Vert D u\Vert_{L^\infty([0,+\infty)\times\T^d)}\leq C$
\item $|D^2u(s,\cdot)|\leq C$ for any $s\in[0,+\infty)$
\item $\bm d(m(s),m(l))\leq C |l-s|^{1/2}$ for any $l,s\in [0,+\infty)$
\end{itemize}

Consequently, we also have that $|\partial_{t}u(s.\cdot)|\leq C$ for any $s\in[0,+\infty)$.
\end{lemma}
\begin{proof}
The proof follows closely the one proposed in \cite{cardaliaguet2017long} and it relies on semiconcavity estimates for the value function $u$. We recall that if $\phi\in C^\infty(\T^d)$ then 
\be\label{uine}
\Vert D\phi\Vert_{L^\infty(\T^d)}\leq d^\frac{1}{2}\sup_{x\in\T^d,\,|\xi|\leq 1} D^2\phi(x)\xi\cdot\xi.
\ee

We first prove the result for $u^{T}:[0,T]\times\T^{d}\rightarrow\R$, solution of 
$$
\begin{cases} 
-\partial_t u-\Laplace u+\delta u+H(x,Du)=F(x,m) & \mbox{in } \T^d\times[0,T]\\
-\partial_t m+\Laplace m+{\rm div}(mD_pH(x,Du))=0 & \mbox{in }\T^d\times[0,T]\\
m(t)=m_0,\; u(T,x)=0&\mbox{in }\T^{d}
\end{cases}
$$
We consider $\xi\in\R^d$, $|\xi|\leq 1$, that maximizes $\sup_{t,x}D^2u^{T}(t,x)\xi\cdot\xi=M$ and we look at the equation solved by $w(t,x)=D^2u^{T}(t,x)\xi\cdot\xi$ deriving twice in space the HJB equation in \eqref{psystem}:
$$
-\partial_t w-\Laplace w +\delta w+D_{\xi\xi}H(x,Du)+2D_{\xi p}H(x,Du)\cdot D^2u\xi
$$
\be\label{D2eq1}
+D_{pp}H(x,Du)D^2u\xi\cdot D^2u\xi+D_pH(x,Du)\cdot Dw=D^2_{\xi\xi}F(x,m).
\ee
The maximum of $w$ can be achieved either at $t=T$, but using the terminal condition of $u^{T}$ we get $M=0$, or at a point $(s,x)$ in the interior. In this case, if we use hypothesis  \eqref{1} and \eqref{2} on $H$ and the boundedness of $D_{xx}^2F$, then, at the maximum $(s,x)$, we get the following inequality
$$
\delta M-C(1+|Du|)^{1+\theta}-2C(1+|Du|)^{\theta}|D^2u\xi|+\bar C^{-1}|D^2u\xi|^2\leq C.
$$
By Cauchy-Schwarz inequality we have that $M=w(s,x)\leq|D^2u(s,x)\xi|$. If we also plug \eqref{uine} we get, for a possible different constant C
$$
-C(1+M)^{1+\theta}-2C(1+M)^{2\theta}+\bar C^{-1}M^2\leq C.
$$

Given that $\theta<1$ the above inequality ensures that $M$ is bounded by a constant that does not depend on $m_0$, $\delta$ and $T$. 
The bound on $\Vert D u^{T}\Vert_{L^\infty([0,T]\times\T^d)}$ follows from \eqref{uine}. Now that we proved that $Du^{T}$ is bounded so that Theorem V 5.4 in \cite{ladyzhenskaia1988linear} gives us the boundedness of $D^2u^{T}$. Note that the estimates on $Du^{T}$ and $D^{2}u^{T}$ imply directly from the HJB equation that $\partial_{t}u^{T}$ is bounded as well. As all the estimates are independent of $T$, if we look at the sequence of $u^{T}$ we have that, on any compact set, $u^{T}$ is uniformly bounded and continuous. This means that $u^{T}$ converges to $u$ solution of the HJB equation on $[0,+\infty)$ and the same estimates hold true for $u$.

Furthermore, it implies that also $D_pH(x,Du)$, the drift of the Fokker Planck equation, is uniformly bounded. Standard results on SDEs (for instance Lemma 3.4 in \cite{cardaliaguet2010notes}) ensure the Holder continuity of $s\mapsto m(s)$ uniformly with respect to $m_0$ and $\delta$.
\end{proof}

 We now fix $m_1^0$, $m_2^0\in\mathcal P(\T^d)$. According to \eqref{dpp} we have
\begin{flalign}
\mathcal V_\delta(m_2^0)-\mathcal V_\delta(m_1^0)\leq
\end{flalign}
\be\label{minU} \int_0^{T} e^{-\delta s}\int_{\T^d}H^*(x,\alpha_2)dm_2(s)-H^*(x,DH(x,Du_1))dm_1(s)+\mathcal F(m_2(s))-\mathcal F(m_1(s))ds,
\ee

where $(u_1,m_1)$ is a solution of \eqref{psystem} related to a minimizer $(m_{1},w_{1})$ of $\mathcal V_{\delta}(m_{1}^{0})$ that we found in Proposition \ref{exmin}. The couple $(m_2,\alpha_2)$ is such that \eqref{fp} is verified in $[0,h+\tau]$ with $m_2(0)=m_2^0$ and $m_2(T)=m_1(T)$. The key point to prove the continuity of $\mathcal V_\delta$ is to construct a suitable $(m_2,\alpha_2)$. We first consider $\tilde m_2$ solution of 
\begin{equation}\label{tildefp}
\begin{cases}
-\partial_t \tilde m_2 +\Laplace \tilde m_2+{\rm div}(\tilde m_2D_pH(x,Du_1))=0 & \mbox{in }\,[0,T]\times\T^d\\
\tilde m_2(0)=m_2^0
\end{cases}
\end{equation}
and then we set
\be\label{m2def}
m_2(s,x)=\begin{cases} 
\tilde m_2(s,x), & \mbox{if }s\in(0,h] \\ 
\frac{\tau+h-s}{\tau}\tilde m_2(s,x)+\frac{s-h}{\tau}m_1(s,x), & \mbox{if }s\in[h,h+\tau]\\
m_1(s,x) & \mbox{if } s\in[\tau+h,T],
\end{cases}
\ee
where $h$ and $\tau$ will be chosen later. Note that, thanks to the boundedness of their drifts, Corollary 6.3.2 in \cite{bogachev2015fokker} ensures that $\tilde m_2$ and $m_1$ have a density for any $s>0$ . What we still need is to define $\alpha_2$ in $[h,h+\tau]$ . We compute the equation verified by $m_2$ in $[h,h+\tau]$ and, using \eqref{psystem} and \eqref{tildefp}, we get
$$
\partial_{t} m_2-\Laplace m_2=\frac{\tau+h-s}{\tau}{\rm div}\left(\tilde m_2 D_pH(x,Du_1)\right)+\frac{s-h}{\tau}{\rm div}\left(m_1 D_pH(x,Du_1)\right)+\frac{m_1-\tilde m_2}{\tau}
$$
that is, by linearity,
$$
\partial_{t} m_2-\Laplace m_2={\rm div}(m_2 D_pH(x,Du_1))+\frac{m_1-\tilde m_2}{\tau}.
$$
Let $\zeta:[0,h+\tau]\times\T^d\rightarrow\R$ be the solution to 
\be\label{pe}
\begin{cases}
\Laplace\zeta=m_1-\tilde m_2, & \mbox{in }[0,h+\tau]\times\T^d\\
\int_{\T^d}\zeta(s,x)=0.
\end{cases}
\ee
We can now define the drift $\alpha_2$ as follows:  $\alpha_2=D_pH(x,Du_1)+\frac{D\zeta}{m_2\,\tau}$ in $[h,h+\tau]$ and $\alpha_2=D_pH(x,Du_1)$ elsewhere. As $(m_2,\alpha_2)$ verifies the Fokker Plank equation by construction with $m_{2}(0)=m_{2}^{0}$, it is admissible. For the continuity of $\mathcal V_\delta$ we still need estimates on the drift $\alpha_2$. We prove those estimates in the next lemma using the regularity of the solutions of the adjoint of the Fokker-Plank equation.
\begin{lemma}\label{duest}
For any time $s<h+\tau$, there exists a constant $K_s>0$, bounded for $s>\varepsilon>0$, such that
$$
\Vert D\zeta(s,\cdot)\Vert_{L^2(\T^d)}\leq K_s \bm d(m_1^0,m_2^0).
$$
The constant $K_{s}$ is independent of $m_1^0$, $m_2^0$.
\end{lemma}
\begin{proof} We first note that, if we multiply \eqref{pe} by $\zeta$ and we use Cauchy-Schwarz inequality on the right hand side, we get
$$
\Vert D\zeta(s)\Vert^{2}_{L^2(\T^{d})}\leq\Vert m_1(s)- \tilde m_2(s)\Vert_{L^{2}(\T^{d})}\Vert \zeta(s)\Vert_{L^2(\T^{d})}.
$$
Now Pointcar\'e-Wirtinger inequality gives us
\be\label{h1emb}
\Vert D\zeta(s)\Vert_{L^2(\T^{d})}\leq C\Vert m_1(s)- \tilde m_2(s)\Vert_{L^{2}(\T^{d})}.
\ee
 
 If we define $\mu(s)=m_1(s)- \tilde m_2(s)$ then $\mu$ verifies the following equation
\be\label{fore}
\begin{cases}
-\partial_t \mu +\Laplace \mu+{\rm div}(\mu D_pH(x,Du_1))=0 & \mbox{in }\,[0,s]\times\T^d\\
\mu(0)=m_1^0-m_2^0&\mbox{in }\T^{d}.
\end{cases}
\ee

We now fix a $\bar\phi\in L^2(\T^d)$ and we consider the adjoint backward equation
\be\label{back}
\begin{cases}
-\partial_t\phi-\Laplace\phi+D_pH(x,Du_1)D\phi=0 & \mbox{in }[0,s]\times\T^{d}\\
\phi(s,x)=\bar\phi(x)&\mbox{in }\T^{d}.
\end{cases}
\ee
Given that $D_pH(x,Du)$ is bounded, if $\phi$ is the solution of \eqref{back}, then there exists a constant $K_s$ (Theorem 11.1 in \cite{ladyzhenskaia1988linear}), such that
\be\label{inte}
\Vert\phi(0)\Vert_{C^{1+\alpha}(\T^d)}\leq K_s \Vert \bar\phi\Vert_{L^2(\T^d)}.
\ee

As the equation \eqref{back} is the adjoint of \eqref{fore},
$$
\int_{\T^d}\phi(s)\mu(s)dx=\int_{\T^d}\phi(0)\mu(0)dx.
$$
We now plug in the initial and terminal conditions and we estimate the righthand side as follows
\be\label{prees}
\int_{\T^d}\bar\phi(x)(m_1(s)(dx)-\bar m_2(s)(dx))=\int_{\T^d}\phi(0,x)(m_1^0(dx)-m_2^0(dx))\leq \Vert D\phi(0)\Vert_{L^\infty}\bm d(m_1^0,m_2^0).
\ee
If we use the interior estimate \eqref{inte} on the righthand side and we take the supremum over $\Vert \bar\phi\Vert_{L^2}\leq 1$, we finally end up with
$$
\Vert m_1(s)-\tilde m_2(s)\Vert_{L^2}\leq K_s\bm d(m_1^0,m_2^0).
$$
If we plug the last inequality into \eqref{h1emb}, we get the result.
\end{proof}

\begin{teo}\label{Ucont}The family of functions $\{\mathcal V_\delta(\cdot)\}_{\delta}$ is uniformly Lipschitz continuous.
\end{teo}
\begin{proof}Let $\alpha_{1}(t,x)=D_pH(x,Du_1)$ for any $t\in[0,h+\tau]$. We consider the same $(m_{2},\alpha_{2})$ that we defined earlier: $m_{2}$ is defined in \eqref{m2def}, $\alpha_{2}=\alpha_{1}+\frac{D\zeta}{m_2\,\tau}$ in $[h,h+\tau]$ and $\alpha_2=\alpha_{1}$ elsewhere, where $\zeta$ solves \eqref{pe}. 
 According to \eqref{minU} we have
\be\label{minU2}
\mathcal V_\delta(m_2^0)-\mathcal V_\delta(m_1^0)\leq\int_0^h e^{-\delta s}\int_{\T^d}H^*(x,\alpha_{1})d(m_2-m_1)+
\ee
\begin{equation*}
+\int_h^{h+\tau}e^{-\delta s}\int_{\T^d}H^*(x,\alpha_2)dm_2-H^*(x,\alpha_{1})dm_1+\int_0^{h+\tau}\mathcal F(m_2(s))-\mathcal F(m_1(s))ds.
\end{equation*}
Using the convexity of $H$, we can estimate the term  $H^*(x,\alpha_2(s))$ for any time $s\in[h,h+\tau]$ as follows
\be\label{hstarcon}
\int_{\T^d}H^*(x,\alpha_2(s))dm_2(s)\leq\int_{\T^d}H^*(x, \alpha_{1}(s))dm_2(s)+ D_pH^*(x,\alpha_2(s))\cdot\frac{D\zeta(s)}{\tau}dx
\ee
$$
\leq \int_{\T^d}H^*(x, \alpha_{1}(s))dm_2(s)+\frac{1}{\tau}|\alpha_{1}(s)||D\zeta(s)|+\bar C\frac{|D\zeta(s)|^2}{\tau^2m_2(s)}dx,
$$
where in the last inequality we add and subtract $\alpha_{1}(s)\cdot D\zeta/\tau$ and we used the growth condition \eqref{1} on $D_{pp}H$.

We recall that, as the drift $D_pH(x,Du_1)$ is continuous and bounded, according to Theorem 2.2.1 \cite{bogachev2009elliptic}, the measure $m_1$ has a density $m_1(s,x) $ for any $s>0$, then, using Theorem 2.5.1 in \cite{bogachev2009elliptic}, for any $l\in(0,s)$, we have 
$$
m_1(s,x)>m_1(l,x_0)e^{-Q(1+\frac{1}{s-l}+\frac{1}{l})},
$$
where $Q$ does not depend on $m_1^0$, $l$ and $s$. As $\T^d$ is bounded, for any $l>0$, there exists a $x_0$ such that $m_1(l,x_0)>1/2$. Given that the same holds true for $\tilde m_2$ then, for any $s\in[h,h+\tau]$, the definition of $m_2$ in \eqref{m2def} implies that
$$
m_2(s,x)>\frac{1}{2}e^{-Q(1+\frac{1}{s-l}+\frac{1}{l})}\qquad\forall l\in(0,h).
$$
For $l=h/2$ we obtain that the infimum, with respect to $s$, in the righthand side is achieved when $s=h$. Thus,
\be\label{m2lb}
m_2(s,x)>\frac{1}{2}e^{-Q(1+\frac{4}{h})}.
\ee
We can now plug \eqref{hstarcon} and \eqref{m2lb} into \eqref{minU2}, which becomes
\be
\mathcal V_\delta(m_2^0)-\mathcal V_\delta(m_1^0)\leq\int_0^{h+\tau} e^{-\delta s}\int_{\T^d}H^*(x,\alpha_{1})d(m_2-m_1)+\mathcal F(m_2(s))-\mathcal F(m_1(s))ds
\ee
\begin{equation*}
+\int_h^{h+\tau}e^{-\delta s}\int_{\T^d}\frac{C_2}{\tau}|D\zeta|+2\frac{C}{\tau^2}|D\zeta|^2e^{Q(1+\frac{4}{h})}dxds.
\end{equation*}
Using the bounds on $Du_1$ found in Lemma \ref{2est}, Lemma \ref{duest} and the regularizing property of $\mathcal F$, we get
$$
\mathcal V_\delta(m_2^0)-\mathcal V_\delta(m_1^0)\leq
$$
\be
 C \int_0^{h+\tau}\bm d(m_2(s),m_1(s))ds+\int_{h}^{h+\tau}C\bm d(m_1^0,m_2^0)K_{s}+2\frac{C}{\tau^2}\bm d^2(m_1^0,m_2^0) e^{Q(1+\frac{4}{h})}K^{2}_sds\leq
\ee
\be
 C \int_0^{h+\tau}\bm d(m_2(s),m_1(s))ds+C\bm d(m_1^0,m_2^0)e^{Q(1+\frac{4}{h})}\int_{h}^{h+\tau}\left(1+\frac{K_{s}}{\tau^{2}}\bm d(m_1^0,m_2^0)\right)K_{s}ds\leq
\ee
\be\label{festimates}
 C \int_0^{h+\tau}\bm d(m_2(s),m_1(s))ds+\frac{C}{\tau^{2}}\bm d(m_1^0,m_2^0)e^{Q(1+\frac{4}{h})}\int_{h}^{h+\tau}K_{s}^{2}ds.
\ee
In the last inequality we neglected the terms which go to infinity slower than $K_{s}^{2}$ and which vanish faster than $\bm d(m_{1}^{0},m_{2}^{0})$. Note that the constant $K_{s}$ might explode when $s$ goes to $0$ but, otherwise, it is bounded. Therefore, as $h>0$, there is no problem of integrability for the term $\int_{h}^{h+\tau}K_{s}^{2}$.

We now focus on the first term in the above inequality. In order to estimate $\bm d(m_1(s),m_2(s))$, we have to look at the SDEs verified by the stochastic processes whose laws are $m_1$ and $m_2$. We first recall that an equivalent formulation of the $1-$Wasserstein distance between two probability measures $\mu$ and $\nu$ is

\be\label{RKd}
\bm d(\mu,\nu)=\inf_{\gamma}\left\{\int_{\T^d\times\T^d}|x-y| d\gamma(x,y)\mbox{ s.t. } \pi_1\gamma=\mu,\,\pi_2\gamma=\nu\right\}.
\ee
We consider a standard probability space $(\Omega,\mathcal G,\mathbb P)$ and two random variables $Z^1$, $Z^2$ such that $\mathcal L(Z^i)=m_i^0$ and $\mathbb E\left[\left|Z^2-Z^1\right|\right]=\bm d(m_1^0,m_2^0)$. Therefore, $m_1$ and $m_2$ are the laws of the processes defined by of the following SDEs
\be
\begin{cases}
dX_s^i=\alpha_i(t,X_s)ds+\sqrt{2}dB_s\\
X^i_0=Z^i.
\end{cases}
\ee
Using the definition of distance in \eqref{RKd}, we have
\be\label{dexp}
\bm d(m_1(s),m_2(s))\leq\mathbb E\left[\left|X^2_s-X_s^1\right|\right]\leq\mathbb E\left[\left|Z^2-Z^1\right|+\int_0^s\left|\alpha_2(l,X_l^2)-\alpha_1(l,X_l^1)\right|dl\right].
\ee
We first split $\int_0^{h+\tau}\bm d(m_2(s),m_1(s))ds$ in the sum of the integrals on the intervals $[0,h]$ and $[h,h+\tau]$. For any $s\in [0,h]$, $\alpha_1(l,x)=D_pH(x,Du_1(t,x))=\alpha_2(l,x)$, then
$$
\bm d(m_1(s),m_2(s))\leq\bm d(m_1^0,m_2^0)+\mathbb E\left[\int_0^s \left|D_pH(X_l^2,Du_1(l,X_l^2))-D_pH(X_l^1,Du_1(l,X_l^1))\right|\right].
$$
Hypothesis \eqref{1}, \eqref{2} and Lemma \ref{2est} ensure that both $p\rightarrow D_pH(x,p)$ and $x\rightarrow D_pH(x,Du_1(l,x))$ are Lipchitz continuous, hence
$$
\bm d(m_1(s),m_2(s))\leq \bm d(m_1^0,m_2^0)+C\int_0^s \bm d(m_1(l),m_2(l))dl.
$$
If we apply Gronwall's inequality, then for any $s\in[0,h]$
\be\label{GWm1}
\bm d(m_1(s),m_2(s))\leq \bm d(m_1^0,m_2^0)e^{Cs}.
\ee
We now look at $\int_h^{h+\tau}\bm d(m_1(s),m_2(s))ds$. According to the definition of $\alpha_2$, for $s\in[h,h+\tau]$, we have 
$$
\bm d(m_1(s),m_2(s))\leq\bm d(m_1(h),m_2(h))+
$$
\be\label{estimates}
+\mathbb E\left[\int_h^s \left|D_pH\left(X_l^2,Du_1(l,X_l^2)\right)+\frac{D\zeta(l,X_l^2)}{\tau m_2(l,X_l^2)}-D_pH(X_l^1,Du_1(l,X_l^1))\right|\right].
\ee
Using \eqref{GWm1} on $\bm d(m_1(h),m_2(h))$ and splitting the last term, we get that \eqref{estimates} is smaller than
$$
\bm d(m_1^0,m_2^0)e^{Ch}+\mathbb E\left[\int_h^s \left|D_pH\left(X_l^2,Du_1(l,X_l^2)\right)-D_pH(X_l^2,Du_1(l,X_l^1))\right|\right]+\mathbb E\left[\int_h^s\left|\frac{D\zeta(l,X_l^2)}{\tau m_2(l,X_l^2)}\right|\right].
$$
If we use again that $x\rightarrow D_pH(x,Du_1(l,x))$ is Lipschitz continuous, we get
\be\label{estimates2}
\bm d(m_1(s),m_2(s))\leq \bm d(m_1^0,m_2^0)e^{Ch}+\frac{C}{\tau}\mathbb E\left[\int_h^s\left|\frac{D\zeta(l,X_l^2)}{m_2(l,X_l^2)}\right|\right]+\int_h^s\bm d(m_1(l),m_2(l)).
\ee
 Thanks to estimates \eqref{m2lb} on $m_2$ we can use Tonelli's theorem and switch the integral with the expectation. Using Lemma \ref{duest}, we eventually have

 $$
 \mathbb E\left[\int_h^s\left|\frac{D\zeta(l,X_l^2)}{m_2(l,X_l^2)}\right|\right]=\int_h^s\int_{\T^d}\left|D\zeta(l,x)\right|dx\leq\bm d(m_1^0,m_2^0)\int_h^s K_ldl.
 $$
 If we plug the last inequality into \eqref{estimates2}, we can use again Gronwall's inequality so that for $s\in[h,h+\tau]$
 \be\label{GWm2}
 \bm d(m_1(s),m_2(s))\leq \left(e^{Cs}+\frac{e^{s-h}}{\tau}\int_h^sK_l dl\right)\bm d(m_1^0,m_2^0).
\ee
 
 We can now suppose $h=\tau=1$ and plugging \eqref{GWm1} and \eqref{GWm2} into \eqref{festimates}, we finally get that, for a given constant $C$ depending on all the other ones 
 \be
 \mathcal V_\delta(m_2^0)-\mathcal V_\delta(m_1^0)\leq
 \ee
 $$
  C\bm d(m_{1}^{0},m_{2}^{0})\left(\int_0^{2}e^{Cs}ds+\int_1^{2} e^{s-1}\int_1^sK_ldlds+e^{5Q}\int_1^{2}K^{2}_sds\right).
 $$
 We recall that the constant $K_s$ of Lemma \ref{duest} is bounded when $h$ is not close to $0$ (Theorem 11.1 in \cite{ladyzhenskaia1988linear}). The infimum in the expression above is finite and none of the constants therein depends on $\delta$. Therefore, $\left\{\mathcal V_{\delta}\right\}_{\delta}$ is uniformly $\bar K$-Lipschitz with 
 $$
 \bar K= C\left(\int_0^{2}e^{Cs}ds+\int_1^{2} e^{s-1}\int_1^sK_ldlds+e^{5Q}\int_1^{2}K^{2}_sds\right).
 $$
 
\end{proof}

\begin{prop}\label{unifconv} For any fixed $\eta\in\mathcal P(\T^d)$ there exists a subsequence $\delta_{n}\to 0$, such that $\mathcal V_{\delta_{n}}(\cdot)-\mathcal V_{\delta_{n}}(\eta)$  uniformly converges to a function $\chi:\mathcal P(\T^d)\rightarrow\R$ when $n\rightarrow +\infty$. Moreover, $\delta_{n} \mathcal V_{\delta_{n}}$ uniformly converges to a constant $-\lambda$
 \end{prop}
\begin{proof} 
The continuity proved in Theorem \ref{Ucont} ensures the boundedness of $\mathcal V_\delta(\cdot)-\mathcal V_\delta(\eta)$. Indeed we have $|\mathcal V_\delta(\cdot)-\mathcal V_\delta(\eta)|\leq \bar K\bm d(\cdot,\eta)$. As $\mathcal P(\T^d)$ is compact, the right hand side is bounded by a constant $K$. Arzel\`a-Ascoli theorem ensures that there exists a subsequence $\delta_{n}\to0$ such that $\mathcal V_{\delta_{n}}(\cdot)-\mathcal V_{\delta_{n}}(\eta)$ converges to a continuous function $\chi$. 

We now want to prove that $\delta \mathcal V_\delta$ is a bounded function. We fix a measure $\mu\equiv 1$, then we define the control $(m,\alpha)$ as follows: $m(t)=\mu$ and $\alpha(t)=0$ for all $t\in[0,+\infty)$. The control is admissible, therefore we have
$$
\delta \mathcal V_{\delta}(\mu)\leq\delta( H^*(x,0)+\mathcal F(\mu))\int_0^{\infty}e^{-\delta s}ds=H^*(x,0)+\mathcal F(\mu).
$$
Given that $H^{*}$ and $\mathcal F$ are bounded from below, then
$$
\delta \mathcal V_{\delta}(\mu)\geq \delta\left(\inf_{(x,a)\in\T^{d}\times\R^{d}}H^{*}(x,a)+\inf_{m\in\mathcal P(\T^{d})}\mathcal F(m)\right)\int_0^{\infty}e^{-\delta s}ds=\inf_{(x,a)}H^{*}(x,a)+\inf_{m}\mathcal F(m).
$$
Therefore, $\delta \mathcal V_{\delta}(\mu)$ is uniformly bounded in $\delta$. If we fix any other measure $m_0$ we can use again the uniform continuity of $\mathcal V_\delta$ to get that $|\delta \mathcal V_{\delta}(m_0)-\delta \mathcal V_{\delta}(\mu)|\leq \delta K$ that in turn tells us that $\delta \mathcal V_\delta (\cdot)$ is a sequence of uniformly continuous functions. Using again Arzel\`a-Ascoli theorem we get that $\delta_{n} \mathcal V_{\delta_{n}}$ uniformly converges to a function $\Psi$ (we can suppose $\delta_{n}$ to be the same subsequence that we identified earlier). Moreover, we have $|\delta_{n} \mathcal V_{\delta_{n}}(\cdot)-\delta_{n} \mathcal V_{\delta_{n}}(\mu)|\leq \delta_{n} K$. Taking the limit we get $|\Psi(\cdot)-\Psi(\mu)|\leq 0$ so that $\delta_{n} \mathcal V_{\delta_{n}}$ converges to the constant function $\Psi(\mu):=-\lambda$.

\end{proof}
\begin{lemma}\label{lem.1.7}Dynamic programming principle for $\chi$: for any $m_{0}\in\mathcal P(\T^{d})$ and $t>0$,
\be\label{dynchi}
\chi(m_{0})= \inf_{(m,\alpha)}\left( \int_{0}^{t}H^{*}(x,\alpha)dm(s)+\mathcal F(m(s))ds+\chi (m(t))\right)+\lambda t
 \ee
 where $m\in C^0([0,t],\mathcal P(\T^d))$, $\alpha\in L^2_{m}([0,t]\times\T^d,\R^d)$ and the pair $(m,\alpha)$ solves in sense of distribution $-\partial_{t}m_+\Laplace m+{\rm div}(m\alpha)=0$ with initial condition $m_{0}$.
\end{lemma}
\begin{proof} In Proposition \ref{unifconv} we proved the convergence of $\mathcal V_{\delta}(\cdot)-\mathcal V_{\delta}(\eta)$ to $\chi(\cdot)$ along the subsequence $\{\delta_{n}\}_{n}$, for a fixed measure $\eta$. Hereafter, $\{\delta_{n}\}_{n}$ and $\eta$ will be the ones identified in that proposition.
 
We know from Proposition \ref{exmin} that, for any $\delta>0$, there exists a solution $(u_{\delta},m_{\delta})$ to \eqref{psystem} such that
\be
\mathcal V_\delta(m_0)=\int_0^t e^{-\delta s}\int_{\T^d}H^*(x,\alpha_\delta)dm_\delta(s)+\mathcal F(m_\delta(s))ds+e^{-\delta t}V_\delta(m_\delta(t)),
\ee
where $\alpha_\delta=D_pH(x,Du_\delta)$. If we take the expansion of $e^{-\delta t}$ and we subtract on both sides $\mathcal V_\delta(\eta)$ we get
\be\label{dynev}
\mathcal V_\delta(m_0)-\mathcal V_\delta(\eta)=\int_0^t e^{-\delta s}\int_{\T^d}H^*(x,\alpha_\delta)dm_\delta(s)+\mathcal F(m_\delta(s))ds+(1-t\delta+o(t\delta))V_\delta(m_\delta(t))-V_\delta(\eta).
\ee
We recall that the estimates in Lemma \ref{2est} are uniform in $\delta$ hence $D_pH(x,Du_{\delta_{n}})$ converges uniformly to a function $\alpha$. We can now take the limit $n\rightarrow +\infty$ and using that $\mathcal V_{\delta_{n}}(\cdot)-\mathcal V_{\delta_{n}}(\eta)\rightarrow\chi(\cdot)$ and $\delta_{n} \mathcal V_{\delta_{n}} (\cdot)\rightarrow -\lambda$, we get
$$
\chi(m_0)=\int_0^t\int_{\T^d}H^*(x,\alpha)dm(s)+\mathcal F(m(s))ds+\lambda t +\chi(m_t)
$$

In order to show that $(\alpha,m)$ is optimal, we fix a competitor $(\beta,\mu)$. According to the dynamic programming principle of $\mathcal V_{\delta}$, if we plug $(\beta,\mu)$ into \eqref{dynev}, we get
$$
\mathcal V_\delta(m_0)-\mathcal V_\delta(\eta)\leq\int_0^t e^{-\delta s}\int_{\T^d}H^*(x,\beta)d\mu(s)+\mathcal F(\mu(s))ds+(1-t\delta+o(t\delta))V_\delta(\mu(t))-V_\delta(\eta).
$$
Taking again the limit on the subsequence $\{\delta_{n}\}_{n}$ we eventually have that
$$
\chi(m_0)\leq\int_0^t\int_{\T^d}H^*(x,\beta)d\mu(s)+\mathcal F(\mu(s))ds+\lambda t +\chi(\mu_t),
$$ 
which proves the result.

\end{proof}
\subsection{Convergence of $\mathcal U(t,\cdot)/t$ and $\delta\mathcal V_{\delta}(\cdot)$}
In this section we propose a Tauberian-type result where we prove that the limit of $\delta\mathcal V_{\delta}(\cdot)$ coincides with the one of $\mathcal U(t,\cdot)/t$ when $t\rightarrow+\infty$.
\begin{teo}\label{utvdconv}The limit value $-\lambda$ is uniquely defined and $\delta\mathcal V_{\delta}(\cdot)\rightarrow -\lambda$ does not depend on a subsequence. Moreover, \label{uconv}$\frac{1}{T}\mathcal U(T,\cdot)$ uniformly converges to $-\lambda$ when $T$ goes to $+\infty$.
\end{teo}
\begin{proof} Let $\{\delta_{n}\}_{n}$ such that $\delta_{n}\mathcal V_{\delta_{n}}\rightarrow-\lambda$ and $\mathcal V_{\delta_{n}}(\cdot)-\mathcal V_{\delta_{n}}(\eta)\rightarrow\chi(\cdot)$. As $\chi$ is a continuous function on the compact set $\mathcal P(\T^{d})$, there exists a constant $C>0$ such that $0\leq \chi(m)+C$ for any $m\in\mathcal P(\T^{d})$. If $(m(t),w(t))\in\mathcal E^{T}_{2}(m_{0})$, then
$$
\int_0^T\int_{\T^d}H^*\left(x,-\frac{dw(t)}{dm(t)}(x)\right)dm(t)+\mathcal F(m(t))dt
$$
$$
\leq\int_0^T\int_{\T^d}H^*\left(x,-\frac{dw(t)}{dm(t)}(x)\right)dm(t)+\mathcal F(m(t))dt+\chi(m(T))+\lambda T-\lambda T+C.
$$
 Taking the infimum over $\mathcal E_{2}^{T}$, the definition of $\mathcal U(T,m_{0})$ and the dynamic programming principle of $\chi$ lead to
 $$
 \mathcal U(T,m_{0})\leq \chi(m_{0})-\lambda T+ C.
 $$
 As the constant $C$ does not depend on $m_{0}$ and $T$, if we divide on both sides by $T$ and we take the limit $T\rightarrow+\infty$, we get
 $$
 \lim_{T\rightarrow+\infty}\frac{1}{T}\mathcal U(T,m_{0})\leq-\lambda.
 $$
The other inequality is analogous. We just need to take a $C_{2}>0$ such that $0\geq \chi(m)-C_{2}$ and repeat the same computation.

Note that the limit $\mathcal U(T,\cdot)/T\rightarrow -\lambda$ is uniform and does not depend on the subsequence $\delta_{n}$ or the function $\chi$. Therefore, the limit $\delta \mathcal V_{\delta}$ is uniquely defined.
\end{proof}

We conclude the section showing that our limit value $\lambda$ is never lower than the ergodic one $\bar\lambda$ defined in \eqref{elambda}.

\begin{prop}\label{mineqlam} Under the above assumptions, $\lambda\geq\bar\lambda$.
\end{prop}
\begin{proof}We know that the convergence of $\mathcal U(\cdot,T)/T$ is uniform, therefore, if $(m,\alpha)$ is an admissible couple for the static problem, we can use it as competitor for $\mathcal U(m,T)$. So,
$$
\frac{1}{T}\mathcal U(T,m)\leq\frac{1}{T}\int_0^T\int_{\T^d}H^*(x,\alpha)dm+\mathcal F(m)dt=\int_{\T^d}H^*(x,\alpha)dm+\mathcal F(m).
$$
If we take the infimum over all admissible static $(m,\alpha)$ we get
$$
\frac{1}{T}\mathcal U(T,m)\leq-\bar\lambda.
$$
Letting $T$ go to $+\infty$, we get the result.
\end{proof}

\subsection{An other representation for $\lambda$}

We can now introduce a third representation for $\lambda$, inspired again by classic results on weak KAM theory (see for instance \cite{fathi1997solutions}), which consists in minimizing over paths with fixed endpoints. 

Let $\Pi_T(m_0,m_1)$ be the set of $(m,\alpha)\in C^0([0,+\infty),\mathcal P(\T^d))\times L^1_{m,loc}([0,+\infty)\times\T^d,\R^d)$ such that $(m,\alpha)$ solves the usual Fokker-Plank equation $-\partial_{t}m+\Laplace m-\dive(m\alpha)=0$ with the extra constraint $m(0)=m_0$ and $m(T)=m_1$. Note that, due to the smoothing property of the parabolic constraint, not for every $m_1$ we can find such a path, so $\Pi_T(m_0,m_1)$ might be the empty set.
\begin{prop}\label{3repr} Let $m_0$, $m_1\in\mathcal P(\T^d)$. If $m_1$ has a density in $H^1(\T^d)$ and there exists an $\varepsilon>0$ such that $m_1>\varepsilon$ almost everywhere then
\be
-\lambda=\lim_{T\rightarrow\infty}\frac{1}{T}\inf_{\Pi_T(m_0,m_1)}\int_0^T\int_{\T^d}H^*(x,\alpha)dm(s)+\mathcal F(m(s))ds.
\ee
\end{prop}
\begin{proof} Let $m_0$ and $m_1$ be as above and $(\bar m,\bar \alpha)$ be optimal for $\mathcal U(T,m_0)$. We extend $(\bar m, \bar w)$ in $[0,T+1]$ as follows: for any $t\in[T,T+1]$ we define $\bar\alpha (t,x)=\bar\alpha(T,x)$ and $\bar m(t,x)$ as the solution of $-\partial_{t} m+\Laplace m+\dive(\bar\alpha m)=0$ with $m(T,x)=\bar m(T,x)$. Note that $\bar\alpha$ is continuous and bounded in $[0,T+1]$, therefore, the estimates \eqref{m2lb} still apply.

We now define a path from $m_0$ to $m_1$ as follows: 
\be
m_2(s,x)=
\begin{cases}
\bar m(s,x)& s\in[0,T]\\
(T+1-s)\bar m(s,x)+(s-T) m_1(x)&s\in[T,T+1]
\end{cases}.
\ee
Let also $\zeta(s,x)$ be solution of $-\Laplace \zeta(s,x)=m_1(x)-\bar m(s,x)$ with $\int_{T^d}\zeta=0$. We can define the control
\be
\alpha_2(s,x)=
\begin{cases}\bar\alpha(s,x)& s\in[0,T]\\
\bar\alpha(s,x)-\frac{(s-T)\bar\alpha(T,x)m_1(x)+D\zeta(s,x)+(s-T)Dm_1(x)}{(T+1-s)\bar m(s,x)+(s-T) m_1(x)}& s\in[T,T+1].
\end{cases}
\ee
The couple $(m_2,\alpha_2)$ belongs to $\Pi_T( m_0,m_1)$.
From the definition of $\mathcal U$ we deduce that
$$
\frac{1}{T+1}\inf_{m\in\mathcal P(\T^d)}\mathcal U(T+1,m)\leq\frac{1}{T+1}\inf_{\Pi_{T+1}(m_0,m_1)}\int_0^{T+1}\int_{\T^d}H^*(x,\alpha)dm(s)+\mathcal F(m(s))ds$$
$$
\leq\frac{1}{T+1}\int_0^{T+1}\int_{\T^d}H^*(x,\alpha_2)dm_2(s)+\mathcal F(m_2(s))ds
$$
\be\label{zerterm}
=\frac{T}{T+1}\left(\frac{1}{T}\mathcal U(T,m_0)\right)+\frac{1}{T+1}\int_T^{T+1}\int_{\T^d}H^*(x,\alpha_2)dm_2(s)+\mathcal F(m_2(s))ds.
\ee
If we prove that $\frac{1}{T+1}\int_T^{T+1}\int_{\T^d}H^*(x,\alpha_2)dm_2(s)+\mathcal F(m_2(s))ds$ converges to zero we have the result. Indeed, if we let $T$ go to $+\infty$, according to Theorem \ref{uconv}, we have 
$$
-\lambda\leq\lim_{T\rightarrow\infty}\frac{1}{T+1}\inf_{\Pi_{T+1}(m_0,m_{1})}\int_0^{T+1}\int_{\T^d}H^*(x,\alpha)dm(s)+\mathcal F(m(s))ds\leq-\lambda.
$$

We now focus on the last part in \eqref{zerterm}. Given that $\mathcal F(m_2)$ is uniformly bounded, we look at the first term.
$$
\int_T^{T+1}\int_{\T^d}H^*(x,\alpha_2)dm_2(s)
$$
$$
\leq C\int_T^{T+1}\int_{\T^d}\frac{|\bar\alpha(T,x) m_{2}(s)+(s-T)\bar\alpha(T,x)m_1(s)+D\zeta(s,x)+(s-T)Dm_1(s)|^2}{m_2^{2}(s)}+1\,dm_{2}(s)ds
$$
\be\label{ine1929}
\leq C\int_T^{T+1}\int_{\T^d}\frac{\left(|\bar\alpha(T,x) m_{2}|+|\bar\alpha(T,x)m_1|+|D\zeta(s,x)|+|Dm_1|\right)^2}{m_2}dxds+C
\ee
If we use the hypothesis on $m_1$ and the estimates \eqref{m2lb} on $\bar m$ with $h=T+1$ and $t=T-1$, we get that $m_2\geq \tau$ for a certain $\tau>0$ independent of $T$. Lemma \ref{2est} ensures that $\bar\alpha$ is uniformly bounded by a constant $K$ independent of $T$. Therefore, \eqref{ine1929} is lower than
$$
\frac{1}{\tau}\left(K\Vert m_{2}\Vert_{L^{2}(\T^{d}\times[T,T+1])}+K\Vert m_{1}\Vert_{L^{2}(\T^{d})}+\Vert D\zeta\Vert_{L^{2}(\T^{d}\times[T,T+1])}+\Vert Dm_{1}\Vert_{L^{2}(\T^{d})}\right)^{2}+C.
$$
Thanks again to the boundedness of $\bar\alpha$, standard result on parabolic equations tell us that $\bar m(s)$ (which is defined at the beginning of the proof) is uniformly bounded from above in $[T,T+1]$. Hence, $\Vert D\zeta\Vert_{L^2}\leq C \Vert m_1\Vert_{L^2}$ and $\Vert m_{2}\Vert_{L^{2}(\T^{d}\times[T,T+1])}\leq C\Vert m_{1}\Vert_{L^{2}(\T^{d})}+C_{2} $. Thus
$$
\int_T^{T+1}\int_{\T^d}H^*(x,\alpha_2)dm_2(s)\leq M \Vert m_1\Vert^2_{H^1(\T^d)}+M_2
$$
where neither $M$ nor $M_2$ depends on $T$. Dividing by $T+1$ and taking the limit completes the proof.
\end{proof}
\section{Projected Mather set and Calibrated curves}
\subsection{Calibrated Curves}

We borrow again some tools and some notations from the weak KAM theory (see Chapter 4 of \cite{fathi2008weak}) and in particular we will focus on the notion of calibrated curve.  Before introducing this notion, we look back to the dynamic programming principle verified by corrector functions, which reads
$$
\chi(m_{0})= \inf_{(m,w)\in\mathcal E^{t}_{2}(m_{0})}\left( \int_{0}^{t}H^{*}\left(x,-\frac{w}{m}\right)dm(s)+\mathcal F(m(s))ds+\chi (m(t))\right)+\lambda t.
$$
As the function $\chi$ is continuous, standard arguments show that, for any fixed $m_{0}\in\mathcal P(\T^{d})$ and $t>0$, there exists a solution $(\bar m,\bar w)\in\mathcal E^{t}_{2}(m_{0})$ to the minimization problem described above. Therefore, extending $(\bar m,\bar w)$ from $\bar m(t)$, it is easy to construct a new trajectory $(\bar m,\bar w)$, defined on $[0,+\infty)$ such that, for any $\tau>0$, it verifies

$$
\chi(m_{0})=\lambda \tau+ \int_{0}^{\tau} \inte H^{*}\left(x, -\frac{\bar w(s)}{\bar m(s)}\right)d\bar m(s) +\mathcal F(\bar m(s))ds+\chi(\bar m(\tau)).
$$
 We now prove that any of these trajectories is associated to a MFG system.

\begin{prop}\label{duplca}
Let $m_{0}\in\mathcal P(\T^{d})$, $\chi$ be a corrector function and $(\bar m,\bar w)$ be a minimizing trajectory on $[0,+\infty)$ defined as above. Then, $\bar{m}\in C^{1,2}((0+\infty)\times\T^{d})$ and there exists a function $\bar u\in C^{1,2}([0,+\infty)\times\T^{d})$ such that $\bar w=-\bar mD_{p}H(x,D\bar u)$ where $(\bar m,\bar u)$ solves
\begin{equation}
\begin{cases} 
-\partial_t u-\Laplace u+H(x,Du)=F(x,m) & \mbox{in } [0,+\infty)\times\T^{d}\\
-\partial_t m+\Laplace m+{\rm div}(mD_pH(x,Du))=0 & \mbox{in }[0,+\infty)\times\T^{d}\\
m(0)=m_0.
\end{cases}
\end{equation}
\end{prop}
\begin{remark}Due to the lack of regularity of $\chi$ we can not derive the MFG system as optimal condition for the minimization problem \eqref{impmin}. Indeed, if $\chi$ were $C^{1}$ we would derive typical MFG system with terminal condition $u(t)=\delta\chi(m(t))/\delta m$ but this latter term is not well defined.
\end{remark}
\begin{proof}The proof relies on the same arguments as in Proposition \ref{exmin}. Let $(\bar m,\bar w)$ be as in the hypothesis. Then it verifies the Fokker-Plank equation and it is a minimizer of the problem
\be\label{impmin}
\inf_{(m,w)\in\mathcal E^{t}_{2}(m_{0})}\int_{0}^{t}\inte H^{*}\left(x,-\frac{w}{m}\right)dm(s)+\mathcal F(m(s))ds+\chi (m(t)).
\ee

As $(\bar m,\bar w)$ is optimal for the minimization problem above, then it must be also optimal for the following MFG planning problem
\be\label{PMFG}
\inf_{(m,w)\in\Pi(m_{0},\bar m(t))}\int_{0}^{t}\inte H^{*}\left(x,-\frac{w}{m}\right)dm(s)+\mathcal F(m(s))ds,
\ee
where $\Pi(m_{0},\bar m(t))$ is the set of $(m,w)\in\mathcal E_{2}^{t}(m_{0})$ that solves the usual Fokker-Plank equation on $[0,t]$ with the constraints $m(0)=m_{0}$ and $m(t)=\bar m(t)$.  We want to prove that $\bar w=-\bar m D_{p}H(x,D\bar u)$ where $(\bar m,\bar u)$ solves in classical sense

 $$
 \begin{cases}
-\partial_t u-\Delta u +H(x,Du)= F(x,m) & {\rm in}\;[0,t]\times \T^d\\
\partial_tm -\Delta m-\dive(mD_pH(x,Du))=0 &{\rm in}\;[0,t]\times \T^d\\
m(0)=m_{0},\, m(t)=\bar m(t).
\end{cases}
$$
We argue again as in Proposition \ref{exmin}. According to Proposition 3.1 in \cite{briani2016stable}, $(\bar m,\bar w)$ minimizes also the following convex problem
\be\label{PMFGc}
\inf_{(m,w)\in\Pi(m_{0}),\bar m(t))}\int_{0}^{t}\inte H^{*}\left(x,-\frac{w}{m}\right)dm(s)+\int_{\T^{d}}F(x,\bar m(s))dm(s)ds.
\ee

This problem admits a dual formulation, in the sense of Fenchel Rokafellar Theorem, which reads
\be\label{dual}
\inf_{\psi\in \bar{\mathcal K}}\left\{\int_{T^{d}}\psi(x,t)d\bar m(t)-\int_{T^{d}}\psi(x,0)dm_{0})\right\}
\ee
where $\bar{\mathcal K}$ is the set of $\psi\in C^{1,2}([0,t]\times\T^{d})$ such that $-\partial_{t}\psi-\Laplace \psi+H(x,D\psi)\leq F(x,\bar m)$. A full justification of the result above can be found again in \cite{cardaliaguet2015second}. 

In the definition of the dual problem we can replace $\bar{\mathcal K}$ with $\mathcal K$, where $\mathcal K$ is the set of $u\in C^{1,2}([0,t]\times\T^{d})$ such that $-\partial_t u-\Delta u +H(x,Du)= F(x,\bar m)$. Indeed, if $\psi$ verifies $-\partial_{t}\psi-\Laplace \psi+H(x,D\psi)\leq F(x,\bar m)$, we can alway consider $u\in C^{1,2}([0,t]\times\T^{d})$ solution of 
$$
 \begin{cases}
-\partial_t u-\Delta u +H(x,Du)= F(x,\bar m) & {\rm in}\;[0,t]\times \T^d\\
u(x,t)=\psi(x,t)&\mbox{in }\T^{d}
\end{cases}
$$
Thanks to the comparison principle we have that $u(0,x)\geq \psi(0,x)$, thus
\be\label{minplan}
\inf_{\psi\in \bar{\mathcal K}}\left\{\int_{T^{d}}\psi(x,0)d\bar m(t)-\int_{T^{d}}\psi(x,0)dm_{0}\right\}\geq\inf_{u\in \mathcal K}\left\{\int_{T^{d}}u(x,t)d\bar m(t)-\int_{T^{d}}u(x,0)dm_{0}\right\}
\ee
The opposite inequality holds by inclusion. 
Lemma 3.2 in \cite{cardaliaguet2015second} and Proposition 3.1 in \cite{briani2016stable}, which rely on the Fenchel-Rokafellar Theorem, ensure that, if $(\bar m,\bar w)$ is a minimizer of \eqref{PMFGc} and $\bar u\in\mathcal K$ is a minimizer of the dual problem, then

$$
\int_{T^{d}}u(x,t)d\bar m(t)-\int_{T^{d}}u(x,0)dm_{0}+\int_{0}^{t}\inte H^{*}\left(x,-\frac{\bar w}{\bar m}\right)d\bar m(s)+\int_{\T^{d}}F(x,\bar m(s))d\bar m(s)ds=0.
$$ 
This implies that $\bar w=-\bar m D_{p}H(x,\bar u)$. As a consequence, we have that $\bar m$ is driven by a smooth drift and so, by Schauder theory, $\bar m\in C^{1,2}((0,t]\times\T^{d})$. In particular, given that $t$ is arbitrary, then $\bar m\in C^{1,2}((0,+\infty)\times\T^{d})$. 

We assumed that the minimization problem \eqref{minplan} admits a solution. The proof  of this result is developed in Lemma \ref{exdualpl} in appendix.
\end{proof}
\begin{remark}\label{justif} Note that the convex duality between the minimization problems
$$
\inf_{(m,w)\in\Pi(m_{0},\bar m(t))}\int_{0}^{t}\inte H^{*}\left(x,-\frac{w}{m}\right)dm(s)+\int_{\T^{d}}F(x,\bar m(s))dm(s)ds
$$ 
and
$$
\inf_{u\in \mathcal K}\left\{\int_{T^{d}}u(x,t)d\bar m(t)-\int_{T^{d}}u(x,0)dm_{0}\right\}
$$
holds true independently from the existence of minimizers for the latter one and, therefore, independently from Lemma \ref{exdualpl}.
\end{remark}

We can now introduce the notion of calibrated curve. Let $\mathcal E_{2}^{\infty}$ be the set of $(m(t),w(t))\in\mathcal P(\T^{d})\times\mathcal M(\T^{d};\R^{d})$ such that $m\in C^0(\R,\mathcal P(\T^d))$, $w$ is absolutely continuous with respect to $m$, its density $dw/dm$ belongs to $L^{2}_{m,loc}(\R\times\T^{d})$ and $-\partial_{t}m+\Laplace m-{\rm div}w=0$ is verified in sense of distribution.

\begin{defin}We say that $(\bar m,\bar w)\in\mathcal E_{2}^{\infty}$ is a calibrated curve if there exists a continuous function $\chi:\mathcal P(\T^{d})\rightarrow \R$ which verifies the dynamic programming principle \eqref{dynchi} and $(\bar m,\bar w)$ is optimal for $\chi$: for any $t_{1}<t_{2}\in\R$
 \be\label{kajehbzredf}
\chi(\bar m(t_1))=\lambda(t_2-t_1)+ \int_{t_1}^{t_2} \inte H^{*}\left(x, -\frac{\bar w(s)}{\bar m(s)}\right)d\bar m(s) +\mathcal F(\bar m(s))ds+\chi(\bar m(t_2)).
\ee

\end{defin}

A direct consequence of Proposition \ref{duplca} is the following result which tells that calibrated curves are smooth and associated to MFG systems defined for any time $t\in\R$.

\begin{prop}\label{duplc}
If $(m,w)\in\mathcal E^{\infty}$ is a calibrated curve, then $m\in C^{1,2}(\R\times\T^{d})$ and there exists a function $u\in C^{1,2}(\R\times\T^{d})$ such that $w=-mD_{p}H(x,Du)$ where $(m,u)$ solves
\begin{equation}
\begin{cases} 
-\partial_t u-\Laplace u+H(x,Du)=F(x,m) & \mbox{in } \R\times\T^{d}\\
-\partial_t m+\Laplace m+{\rm div}(mD_pH(x,Du))=0 & \mbox{in }\R\times\T^{d}
\end{cases}
\end{equation}
\end{prop}

\subsection{The projected Mather set}

 \begin{defin}We say that $m_0\in \mathcal P(\T^{d})$ belongs to the \textit{projected Mather set} ${\mathcal M}\subset \mathcal P(\T^{d})$ if there exists a calibrated curve $(m(t),w(t))$ such that $m(0)=m_{0}$. 
\end{defin}
Note that, if from $m_{0}$ starts a calibrated curve $m(t)$, then, by translation, $m(t)\in \mathcal M$ for any $t\in\R$.

\begin{prop}\label{nemptms} There exists a calibrated curve and, consequently, the projected Mather set $\mathcal M$ is not empty.
\end{prop}
\begin{proof}We fix a smooth density $m_{0}$ and we look at the $\delta$-discounted problem \eqref{discprob} which reads
\be
\mathcal V_\delta(m_0)=\inf_{(m,w)}\int_0^\infty e^{-\delta t}\int_{\T^d}H^*\left(x,-\frac{w}{m}\right)dm(t)+\mathcal F(m(t))dt.
\ee
We recall that $\mathcal V_{\delta}$ satisfies the dynamic programming principle
\be
\mathcal V_\delta(m_0)=\inf_{m,w}\int_0^T e^{-\delta s}\int_{\T^d}H^*\left(x,-\frac{w}{m}\right)dm(s)+\mathcal F(m(s))ds+e^{-\delta T}\mathcal V_\delta(m(t)),
\ee
where the infimum is taken over $(m,w)\in\mathcal E_{2}^{\delta}(m_{0})$. We already know that the solution of the minimization problem corresponds to a couple $(\bar m_{\delta}^{T},-\bar m_{\delta}^{T}D_{p}H(x,D\bar u_{\delta}^{T}))$ where $(\bar m_{\delta}^{T},\bar u_{\delta}^{T})$ solves
\begin{equation}\label{nsconv}
\begin{cases} 
-\partial_t u-\Laplace u+\delta u+H(x,Du)=F(x,m) & \mbox{in } \T^d\times[0,+\infty)\\
-\partial_t m+\Laplace m+{\rm div}(mD_pH(x,Du))=0 & \mbox{in }\T^d\times[0,+\infty)\\
m(0)=m_0&\mbox{in }\T^{d}.
\end{cases}
\end{equation}
Note that, as the initial condition is smooth, the solution $(\bar m_{\delta}^{T},\bar u_{\delta}^{T})$ is smooth as well.

We define the new couple $(m_{\delta}^{T},w_{\delta}^{T})$ as $m_{\delta}^{T}(t,x)=\bar m_{\delta}^{T}(t+T,x)$ and $w_{\delta}^{T}(t,x)=\bar w_{\delta}^{T}(t+T,x)$ so that our problem is set on $[-T,+\infty)$. We now want to prove that, when we take the limit $T\rightarrow+\infty$, our sequence converges to a couple $(m_{\delta},w_{\delta})$ defined on $\R\times\T^{d}$ such that the Fokker-Plank equation is still verified. We proved in Lemma \ref{2est} that the drift $D_{p}H(x,Du_{\delta}^{T})$ is uniformly bounded in $T$, therefore, $m^{T}_{\delta}$ is the solution of a Fokker-Plank equation with bounded and smooth drift. This means that $m^{T}_{\delta}$ is uniformly bounded in $C^{1,2}([-T+1,+\infty)\times\T^{d})$. 

This implies that, at least on compact subsets of $\R\times\T^{d}$, when we take the limit $T\rightarrow+\infty$, we have, up to a subsequence, uniform convergence of $m^{T}_{\delta}$ to a limit $m_{\delta}$. 

The same convergence holds true also for $w^{T}_{\delta}$. Indeed, in Lemma \ref{2est} we proved also the uniform boundedness of $D^{2}u^{T}_{\delta}$ and $\partial_{t}u^{T}_{\delta}$ that implies the uniform continuity and the uniform boundedness of $w^{T}_{\delta}$. The convergence $(m_{\delta}^{T},w_{\delta}^{T})$ to $(m_{\delta},w_{\delta})$ ensures that the couple $(m_{\delta},w_{\delta})$ verifies the Fokker-Plank equation on $\R\times\T^{d}$.

 We fix two different times $t_{1}<t_{2}$. For sufficiently large $T$, the interval $[t_{1},t_{2}]$ is included in $[-T,+\infty)$. If we apply the dynamic programming principle for  $\mathcal V_{\delta}$, we get
 \be
\mathcal V_\delta(m^{T}_{\delta}(t_{1}))=\int_{t_{1}}^{t_{2}} e^{-\delta (s-t_{1})}\int_{\T^d}H^*\left(x,-\frac{w_{\delta}^{T}}{m_{\delta}^{T}}\right)dm_{\delta}^{T}(s)+\mathcal F(m_{\delta}^{T}(s))ds+e^{-\delta (t_{2}-t_{1})}\mathcal V_\delta (m_{\delta}^{T}(t_{2})).
\ee

We can now take the limit of $T\rightarrow+\infty$ in the above expression and we find that $(m_{\delta},w_{\delta})$ verifies
\be
\mathcal V_\delta(m_{\delta}(t_{1}))=\int_{t_{1}}^{t_{2}} e^{-\delta (s-t_{1})}\int_{\T^d}H^*\left(x,-\frac{w_{\delta}}{m_{\delta}}\right)dm_{\delta}(s)+\mathcal F(m_{\delta}(s))ds+e^{-\delta (t_{2}-t_{1})}\mathcal V_\delta (m_{\delta}(t_{2})).
\ee
for any $t_{1}<t_{2}\in\R$.

As the function $u_{\delta}^{T}$ is uniformly bounded in $T$ we have also uniform convergence of $u_{\delta}^{T}$ to a function $u_{\delta}$. We can then pass to the limit in the MFG system \eqref{nsconv} and the couple $(u_{\delta},m_{\delta})$ solves

\begin{equation}\label{nsconv2}
\begin{cases} 
-\partial_t u_{\delta}-\Laplace u_{\delta}+\delta u_{\delta}+H(x,Du_{\delta})=F(x,m_{\delta}) & \mbox{in } \T^d\times\R\\
-\partial_t m_{\delta}+\Laplace m_{\delta}+{\rm div}(m_{\delta}D_pH(x,Du_{\delta}))=0 & \mbox{in }\T^d\times\R.
\end{cases}
\end{equation}

As in \cite{barles2001space}, in order to let $\delta\rightarrow 0$, we need to define $\bar u_{\delta}(t,x)=u_{\delta}(t,x)-u_{\delta}(0,0)$ and $\theta_{\delta}=u_{\delta}(0,0)$. The couple $(\bar u_{\delta},m_{\delta})$ solves

\begin{equation}\label{nsconv3}
\begin{cases} 
-\partial_t \bar u_{\delta}-\Laplace \bar u_{\delta}+\delta \bar u_{\delta}+\delta\theta_{\delta}+H(x,D\bar u_{\delta})=F(x,m_{\delta}) & \mbox{in } \T^d\times\R\\
-\partial_t m_{\delta}+\Laplace m_{\delta}+{\rm div}(m_{\delta}D_pH(x,D\bar u_{\delta}))=0 & \mbox{in }\T^d\times\R\\
\bar u_{\delta}(0,x)=u_{\delta}(0,x)-u_{\delta}(0,0)&\mbox{in }\T^{d}.
\end{cases}
\end{equation}

We restrict ourselves to the subsequence $\{\delta_{n}\}_{n}$ identified in the proof of Lemma \ref{lem.1.7}. Using again the uniform estimates on $D\bar u_{\delta}$, we have that $\bar u_{\delta}(0,x)$ is uniformly bounded which implies the boundedness of $\delta\theta_{\delta}$. Moreover, thanks to the bounds on $D^{2} u_{\delta}$ and $\partial_{t}u_{\delta}$, $\bar u_{\delta}$ is also uniformly continuous and the same holds true for $m_{\delta}$. We can then pass to the limit on any compact set and $\bar u_{\delta_{n}}\rightarrow u$, $m_{\delta_{n}}\rightarrow m$ and $\delta_{n}\theta_{\delta_{n}}\rightarrow \theta$ where $(u,m,\theta)$ solves

\begin{equation}\label{nsconv4}
\begin{cases} 
-\partial_t u-\Laplace u+\theta+H(x,D u)=F(x,m) & \mbox{in } \T^d\times\R\\
-\partial_t m+\Laplace m+{\rm div}(mD_pH(x,Du))=0 & \mbox{in }\T^d\times\R\\
\end{cases}
\end{equation}

As we can always replace $u(t,x)$ with $u(t,x)-\theta t$ we can suppose $\theta=0$. 
The convergences above give us also the uniform convergence on compact sets (up to subsequence) of the couple $(m_{\delta_{n}},w_{\delta_{n}})=(m_{\delta_{n}},-m_{\delta_{n}}D_{p}H(x,D\bar u_{\delta_{n}}))$ to $(m,w)=(m,-m D_{p}H(x,Du))$ which solves the usual Fokker-Plank equation.

 Let now $\eta\in\mathcal P(\T^{d})$ be the measure identified in the proof of Lemma \ref{lem.1.7}. Then 

$$
\mathcal V_{\delta_{n}}(m_{\delta_{n}}(t_{1}))-\mathcal V_{\delta_{n}}(\eta)=
$$
$$
\int_{t_{1}}^{t_{2}} e^{-\delta_{n}(s-t_{1})}\int_{\T^d}H^*\left(x,-\frac{w_{\delta_{n}}}{m_{\delta_{n}}}\right)dm_{\delta_{n}}(s)+\mathcal F(m_{\delta_{n}}(s))ds+e^{-\delta_{n} (t_{2}-t_{1})}\mathcal V_{\delta_{n}} (m_{\delta_{n}}(t_{2}))-\mathcal V_{\delta_{n}}(\eta).
$$
Given the continuity of $\mathcal V_{\delta}$, the uniform convergence of $m_{\delta}$ and $w_{\delta}$, we can pass to the limit in $n$ and we finally get that for any interval $[t_{1},t_{2}]$ the couple $(m,w)$ verifies the Fokker-Plank equation on $\R$ and

$$
\chi(m(t_{1}))= \int_{t_{1}}^{t_{2}}\inte H^{*}\left(x,-\frac{w}{m}\right)dm(s)+\mathcal F(m(s))ds+\chi (m(t_{2}))+\lambda (t_{2}-t_{1}).
$$

In particular we found a calibrated curve and, for any $t\in\R$, $m(t)$ belongs to the projected Mather set $\mathcal M$.

\end{proof}

\subsection{Compactness of the projected Mather set}

In Proposition \ref{duplca} we proved that, if $\chi$ is a corrector function and $(m,w)$ is a trajectory starting from $m_{0}\in\mathcal P(\T^{d})$ which is optimal for the dynamic programming principle of $\chi$, then $(m,w)$ is associated to a MFG system which enjoys the estimates we proved in Lemma \ref{2est}. Therefore, a completely analogous proof to the one proposed in Theorem \ref{Ucont} gives the following result.

 \begin{prop}\label{remlip}
  The set of corrector functions is uniformly Lipschitz continuous.
\end{prop}

We can now prove the compactness of the projected Mather set $\mathcal M$
\begin{prop}The projected Mather set $\mathcal M$ is compact 
\end{prop}
\begin{proof} 
 
 Let $m_{n}\in\mathcal M$ such that $m^{0}_{n}\rightarrow m_{0}$. Let $(m_{n}(t),w_{n}(t))$ be the calibrated curve starting from $m^{0}_{n}$. For any $t_{1}$, $t_{2}$ we know that $(m_{n}(t),w_{n}(t))$ verifies
\be\label{chidef}
\chi_{n}(m_{n}(t_{1}))\geq\int_{t_{1}}^{t_{2}}\inte H^{*}\left(x,-\frac{w_{n}}{m_{n}}\right)dm_{n}(s)+\mathcal F(m_{n}(s))+\lambda (t_{2}-t_{1})+\chi_{n}(m_{n}(t_{2})).
\ee

We know from Proposition \ref{remlip} that the set $\{\chi_{n}\}_{n}$ is uniformly Lipschitz. If we replace $\chi_{n}$ with $\chi_{n}(\cdot)-\chi_{n}(\eta)$, then $\{\chi_{n}\}_{n}$ is also bounded and thus compact. Therefore, we can pick a subsequence such that $\chi_{n}$ converges to a function $\chi$.

Given that $\chi_{n}$ are uniformly bounded, then 
$$
\int_{t_{1}}^{t_{2}}\inte H^{*}\left(x,-\frac{w_{n}}{m_{n}}\right)dm_{n}(s)+\mathcal F(m_{n}(s))ds\leq C.
$$
The constant $C$ does not depend on $n$ and, therefore, $\int_{t_{1}}^{t_{2}}|w_{n}|$ is uniformly bounded as well. As we argued in Proposition \ref{exmin}, this implies that $\{m_{n}\}_{n}$ is uniformly bounded in $C^{1/2}([t_{1},t_{2}],\mathcal P(\T^{d}))$ for any $t_{1}$, $t_{2}$. We have then that $m_{n}$ converges uniformly on any compact set to a limit $m\in C^{0}(\R,\mathcal P(\T^{d}))$ and the same holds true for $w_{n}$ in $\mathcal M (\R\times\T^{d};\R^{d})$, therefore $(m,w)$ solves in sense of distribution the usual FP equation on $\R$. By weak lower semicontinuity of the integral part in \eqref{chidef} and the uniform convergence of $\chi_{n}$ we get that 
$$
\chi(m(t_{1}))\geq \int_{t_{1}}^{t_{2}}\inte H^{*}\left(x,-\frac{w}{m}\right)dm(s)+\mathcal F(m(s))ds+\chi (m(t_{2}))+\lambda (t_{2}-t_{1})
$$
with $m(0)=m_{0}$ because $m^{0}_{n}\rightarrow m_{0}$. The opposite inequality is true by dynamic programming principle and so this proves that $m_{0}\in\mathcal M$ and, eventually, that $\mathcal M$ is closed.

\end{proof}

\subsection{Minimal invariant set and Ergodicity}

We say that a closed subset ${\mathcal C}$ of ${\mathcal M}$ is invariant if, for any $m_0\in\mathcal C$,  there exists a calibrated curve $(m,w)$ such that $m(0)=m_0$ and $m(t)\in{\mathcal C}$ for any $t\in\R$. We say that an invariant set  ${\mathcal C}$ is minimal if  ${\mathcal C}$ does not contains any proper closed invariant subset.

\begin{lemma}\label{mininvset} There exists a minimal set $\mathcal N$.
\end{lemma}

We do not present the proof which is a standard application of Zorn's Lemma (see for instance \cite{de2013elements}).

\begin{prop}\label{mdensity} Let $\mathcal N$ be a minimal invariant set. If $m_{0}\in\mathcal N$ and $\{m(t),\;t\in\R \}$ is a calibrated curve such that $m(0)=m_{0}$, then $\{m(t),\;t\in\R \}$ is dense in $\mathcal N$.
\end{prop}
\begin{proof}The proof is analogous to the one we used to prove that $\mathcal M$ is closed. We define $\mathcal C$ the closure of the trajectory $\{m(t)\;t\in\R \}$. $\mathcal C$ is a closed subset of $\mathcal N$, in order to prove that it coincides with $\mathcal N$ we just need to prove that it is invariant or, in other words, that, if $\bar m\in\mathcal C$, then also a calibrated curve passing through $\bar m$ belongs to $\mathcal C$.

 Let $\bar m$ be the limit of $m_{n}=m(t_{n})\in\{m(t),\;t\in\R \}$ and $\{m_{n}(t)\}$ their corresponding calibrated curves. If $w_{n}(t)$ are the control associated to the calibrated curve $m_{n}(t)$ then, as in Lemma \ref{mininvset}, we get that $\Vert w_{n}\Vert_{L^{1}}$ is uniformly bounded on any compact set. As we already pointed out, it implies the uniform convergence of $m_{n}(t)$ on compact sets.  If $\bar m(t)$ is the trajectory to which $m_{n}(t)$ converges, then it must be a calibrated curve starting from $\bar m$ because we imposed that $m_{n}\rightarrow\bar m$. This means that for any $s\in\R$, $\bar m(s)$ is the limit of $m_{n}(s)$. As $\mathcal C$ is closed and $m_{n}(s)\in\mathcal C$, then $\bar m(s)\in\mathcal C$. 
 
 We proved that $\mathcal C$ is a not empty, invariant, closed subset of $\mathcal N$. By the minimality of $\mathcal N$ the two sets must conicide.
\end{proof}

\section{The role of Monotonicity}
So far, the hypothesis on $\mathcal F$ were mostly about its regularity and no structural assumptions were imposed. On the other hand, when we are interested in understanding whether the limit value $\lambda$ coincides with $\bar\lambda$, the ergodic one, the structure of $\mathcal F$ does actually play a fundamental role. In the next section we impose convexity and, as it was already proved in \cite{cardaliaguet2013long2}, we get that $\lambda=\bar\lambda$. More interestingly, in Section \ref{sec2} we provide a class of explicit examples where $\lambda>\bar\lambda$ and, therefore, there is not convergence of the time dependent MFG system to the ergodic one.

\subsection{The convex case: $\lambda=\bar\lambda$}

 In this section we will show that under convexity those two values are the same. 

We assume the following further assumptions on $F$:
\be\label{monot}
\int_{\T^d} (F(x,m_1)-F(x,m_2))d(m_1-m_2)\geq0
\ee 

We introduce the functional $J^{T}(m_{0},\cdot,\cdot):\mathcal E^{T}(m_{0})\rightarrow\R$ defined by
\be
 J^{T}(m_{0},m,w)=\int_0^{T} \int_{\T^d}H^*\left(x,-\frac{dw(t)}{dm(t)}(x)\right)dm(t)+\mathcal F(m(t))dt,
 \ee
 so that
 $$
\mathcal U(T,m_0)=\inf_{(m,w)\in\mathcal E^{T}_{2}(m_{0})}J^{T}(m_{0},m,w).
 $$

Under the monotonicity assumption \eqref{monot}, the functional $J^T$ is convex, therefore we can easily prove that $\lambda=\bar\lambda$. We recall that 
$$
-\bar\lambda=\inf_{(m,w)\in\mathcal E}\int_{\T^d}H^*\left(x,-\frac{dw}{dm}(x)\right)dm(x)+\mathcal F(m).
$$

\begin{prop}Under the above assumptions $\lambda=\bar\lambda$.
\end{prop}
\begin{proof} In order to prove the proposition we use the representation of $\lambda$ that we discussed in Proposition \ref{3repr}:
\be
-\lambda=\lim_{T\rightarrow\infty}\frac{1}{T}\inf_{\Pi_T(m_0,m_1)}\int_0^T\int_{\T^d}H^*(x,\alpha)dm(s)+\mathcal F(m(s))ds.
\ee

Given that $\lambda$ does not depend on the initial value $m_{0}$, we take $m_{0}=m_{1}$. We can also suppose that $m_{0}$ is smooth and bounded from below by a positive constant, so that we can apply Proposition \ref{3repr}. We now consider any admissible $(m^T,w^T)$ for $\Pi(m_{0},m_{0})$ and we define $\bar m^T=\frac{1}{T}\int_0^Tm^Tdt$ and $\bar w^T=\frac{1}{T}\int_0^T w^Tdt$. Given that the $m^T(0)=m^T(T)=m_{0}$, the couple $(\bar m^T,\bar w^T)$ verifies, in sense of distributions, $-\Laplace m+{\rm div}(w)=0$ and it is an admissible competitor for the stationary problem. 

Now we just need to apply Jensen's inequality to get

$$
J^T(m_{0},m^T,w^T)=\frac{1}{T}\int_0^T\int_{\T^d}H^*(x,w^T/m^T)dm^T+\mathcal F(m^T)dt\geq \int_{\T^d}H^*(x,\bar w^T/\bar m^T)d\bar m^T+\mathcal F(\bar m^T).
$$
If we take the infimum over $(m^{T},w^{T})\in\Pi_{T}(m_{0},m_{0})$ and we take the limit in $T$, we end up with
$$
-\lambda\geq\lim_{T\rightarrow\infty}\inf_{\Pi_T(m_0,m_0)}\int_{\T^d}H^*(x,\bar w^T/\bar m^T)d\bar m^T+\mathcal F(\bar m^T)\geq-\bar\lambda.
$$

We already proved the opposite inequality in Proposition \ref{mineqlam}, which we recall that it holds true also outside the monotonicity assumption.
\end{proof}

As $\lambda=\bar\lambda$, the dynamic programming principle for $\chi$ reads 

$$
\chi(m(t_{1}))= \int_{t_{1}}^{t_{2}}\inte H^{*}\left(x,-\frac{w}{m}\right)dm(s)+\mathcal F(m(s))ds+\chi (m(t_{2}))+\bar\lambda (t_{2}-t_{1}).
$$
If $(\bar m,\bar w)$ is a minimizer for the static MFG problem then, if we define $(m(t),w(t))=(\bar m,\bar w)$ for any $t\in\R$, the relation above holds true. This means that the constant trajectory $(\bar m,\bar w)$ is a calibrated curve so that $\bar m\in\mathcal M$. Moreover, as the calibrated curve is stationary, the singleton $\mathcal N=\{\bar m\}$ is a minimal invariant set because it is closed, invariant and it cannot contain any proper subset.

\subsection{A non convex example where $-\lambda<-\bar\lambda$}\label{sec2}

We now present an example where the non convexity of $\mathcal F$ leads to an ergodic configuration where the limit value $-\lambda$ is strictly lower then the ergodic one $-\bar\lambda$. A straightforward consequence will be that there cannot be any stationary calibrated curve, which in turn implies that any calibrated curve in a minimal invariant set has to be either periodic or chaotic. 
We say that a calibrated curve $m(t)$ in a minimal invariant set $\mathcal N$ has a chaotic behavior if its trajectory is strictly included in $\mathcal N$.

We fix $e_{d}\in\R^{d}\setminus\{0\}$ a unit vector parallel to one of the axes and we identify $\T^{d}$ with $\T^{d-1}\times\T$ where $\T$ is the torus identified by the direction $e_{d}$. We fix also $H$ such that  $H^{*}$ verifies the following conditions:  $H^{*}(x,p)>0$ for any $x\in\T^{d}$, $p\neq-e_{d}$ and $H^{*}(x,-e_{d})=0$ for any $x\in\T^{d}$.
The assumption \eqref{1} on the hamiltonian $H$ implies that there exist two constants $C_{1}>0$, $C_{2}>0$ such that
\begin{align}\label{gcond2}
C_{1}I_{d}\leq D_{pp}H^{*}(x,p)\leq C_{2}I_{d}\qquad\qquad\forall x\in\T^{d},\,\forall p\in\R^{d}.
\end{align}

Let us define the set $\mathcal A\subset\mathcal P(\T^{d})$ as the set of $\mu\in\mathcal P(\T^{d})$ for which there exits $\mu'\in\mathcal P(\T^{d-1})$ such that $\mu=\mu'\otimes dx_{d}$, where $dx_{d}$ si the Lebesgue measure on $\T$.
Note that $\mu\in\mathcal A$ if and only if ${\rm div}(e_{d}\mu)=0$.

 We fix $m_0:\T^{d}\rightarrow\R$ a smooth, strictly positive density such that $m_{0}\notin \mathcal A$. A measure $m$ belongs to the set $\mathcal B$ if there exists $z\in\T^{d}$ such that $m(\cdot)=m_{0}(\cdot+z)$. As $\mathcal A$ and $\mathcal B$ are closed and disjoint, they are separated by a positive distance $\varepsilon>0$.
 
 We choose $\mathcal F:\mathcal P(\T^{d})\rightarrow\R$ such that $\mathcal F\geq 0$, $\mathcal F\equiv 2$ in $\mathcal A$ and $\mathcal F\equiv 0$ in $\mathcal B$. The existence of such a function is ensured by Lemma \ref{199} and Lemma \ref{200} in Appendix. They also guarantee that we can choose $\mathcal F$ such that it verifies the regularity assumptions that were in place in the previous sections.
 
We recall that the functional $J^{T}(\mu,\cdot,\cdot)$ is defined on $\mathcal E_{2}^{T}(\mu)$ by
\be
 J^{T}(\mu,m,w)=\int_0^{T} \int_{\T^{d}}H^{*}\left(x,-\frac{dw(t)}{dm(t)}(x)\right)dm(t)+\mathcal F(m(t))dt.
 \ee
 
In this section we add in the definition of $\mathcal E_{2}^{T}(\mu)$ a viscosity constant $\sigma>0$ so that $(m,w)$ verifies $-\partial_{t}m+\sigma\Laplace m-{\rm div}(w)=0$.

We also define the ergodic functional $J:\mathcal E\rightarrow\R$ as follows
\be\label{staticfunc2}
J(m,w)= \int_{\T}H^{*}\left(x,-\frac{dw(t)}{dm(t)}(x)\right)dm(t)+\mathcal F(m(t))dt,
\ee
 where in this framework, $(m,w)$ verifies $\sigma\Laplace m-{\rm div} w=0$. According to the definition of $\bar\lambda$ in \eqref{staticfunc3}, we have
\be\label{staticfunc5}
-\bar\lambda=\inf_{(m,w)\in\mathcal E}J(m,w)
\ee

\begin{prop} There exists a $\sigma_{0}>0$ such that for any $\sigma\in(0,\sigma_{0}]$ we have $-\lambda<-\bar\lambda$.
\end{prop}
\begin{proof}
 We define 
\begin{align}
m (t,x)=m_0(x-e_{d} t),
\end{align}
and 
\begin{align}
w(t,x)=e_{d} m(t,x)+\sigma Dm_{0}(x-e_{d} t).
\end{align}
The couple $(m, w)$ belongs to $\mathcal E_{2}^{T}(m_{0})$, so $-\lambda\leq J^{T}( m, w)$. By definition of $\mathcal{F}$, we know that $\mathcal{F}(m(t))=0$ for any time $t$. Moreover, since $D_{pp}L\leq C_2I_d$ with $0= H^{*}(x,-e_{d})\leq L(x,\alpha)$, we have $H^{*}(x,\alpha)\leq \frac{1}{2}C_2|\alpha+e_{d}|^2$. Thus 
$$
-\lambda  \leq\frac{1}{T}\int_{0}^{T}\int_{\T^{d}}\frac{C_2}{2} \left|\frac{ w(t,x)}{m(t,x)}-e_{d}\right|^2 m(t,x) dxdt = \frac{C_2}{2}\sigma^2 \int_{\T^{d}} \frac{|Dm_0(x)|^2}{m_0(x)} dx=\sigma^{2}I,
$$
where $I=\frac{C_2}{2} \int_{\T^{d}} \frac{|Dm_0(x)|^2}{m_0(x)} dx$.

We now focus on the static case. We recall that the differential constraint on $J$ is $-\sigma \Laplace m+{\rm div}w=0$. By standard arguments we have that there exists a minimizer $(\bar m,\bar w)$ of \eqref{staticfunc2}.  

As in the proof of Proposition \ref{exmin} we can define a dual problem which reads 
$$
\inf_{(u,\lambda)\in C^{2}(\T^{d})\times\R}\left\{\lambda\mbox{ s.t.}-\lambda-\Laplace u+H(x,Du)\leq F(x,\bar m)\right\}.
$$
 Thanks to the regularity of $F(\cdot,\bar m)$ we have a smooth solution $(\bar u,\bar\lambda)$ which solves $-\bar\lambda+\Laplace \bar u+H(x,D\bar u)=F(x,\bar m)$. By duality, if we argue as in Proposition \ref{exmin} (see \cite{briani2016stable}) we get that $\bar w=-\bar mD_{p}H(x,D\bar u)$, so that, by Schauder theory, $\bar m$ is smooth and bounded from below.

We can now estimate $-\bar\lambda$. 
Thanks to the regularity of $(\bar m,\bar w)$, the parabolic constraint ensures that $\bar w=\sigma D\bar m+\zeta$ where $\zeta$ is a smooth, divergence free vector field. If $(\bar m, \bar w)=( \bar m, \sigma D\bar m+ \zeta)$ is a minimizer of \eqref{staticfunc2} and we use the growth assumption \eqref{gcond2}, we have

\begin{align}\label{lastine}
\int_{\T^{d}}H^{*}\left(x,\frac{\sigma D\bar m+\zeta}{\bar m}\right)\bar m(dx)+\mathcal F(\bar m)\geq \frac{C_{1}}{2}\int_{\T^{d}}\left|\frac{\sigma \bar Dm+\zeta}{\bar m}+e_{d}\right|^2\bar m(dx)+{\mathcal F}(\bar m).
\end{align}

Note that, as ${\rm div}\zeta=0$ and $\bar m$ is smooth and bounded from below, $\int_{\T^{d}}\zeta\cdot D\bar m/\bar m=0$. Indeed,
$$
\int_{\T^{d}}\frac{D\bar m}{\bar m}\zeta dx=\int_{\T^{d}}D(\ln (\bar m))\zeta dx=-\int_{\T^{d}}\ln(\bar m){\rm div}(\zeta) dx=0.
$$
Therefore, if we expand the square in \eqref{lastine}, we get 
\be\label{aaaa}
\int_{\T}H^{*}\left(x,\frac{\sigma D\bar m+\zeta}{\bar m}\right)\bar m(dx)+{\mathcal F}(\bar m) \geq \frac{C_{1}}{2}\int_{T^{d}}\sigma^{2}\frac{|D\bar m|^{2}}{\bar m}+\left|\frac{\zeta}{\bar m}+e_{d}\right|^{2}\bar mdx\geq \frac{C_{1}}{2}\int_{T^{d}}\left|\frac{\zeta}{\bar m}+e_{d}\right|^{2}\bar mdx.
\ee
Plugging \eqref{aaaa} into \eqref{staticfunc5} we eventually find that

$$
-\bar\lambda\geq\frac{C_{1}}{2}\int_{T^{d}}\left|\frac{\zeta}{\bar m}+e_{d}\right|^{2}\bar mdx+\mathcal F(\bar m).
$$
The righthand side of the above inequality is bounded from below by a positive constant independent of $\sigma$. Indeed, we know that, for any $\sigma$, $\bar m>0$, so
$$
-\bar\lambda\geq\inf_{(m,\xi)}\int_{T^{d}}\left|\frac{\xi}{ m}+e_{d}\right|^{2} mdx+\mathcal F(m)
$$
where the infimum is taken over all the probability densities $m>0$ and the free divergence vectors $\xi$. Here, $m$ and $\xi$ do not verify the elliptic constraint, therefore we lose the dependence on $\sigma$. 

Let $(m_{n},\xi_{n})$ be a minimizing sequence and $m\in\mathcal P(\T^{d})$ a the limit of $m_{n}$ (the existence of $m$ is guaranteed by the compactness of $\mathcal P(\T^{d})$). If the infimum were achieved at zero then ${\rm div}(m_{n}e_{d})\rightarrow \dive(me_{d})=0$. Indeed, as both the addends should converge to zero, for any test function $\varphi$, we have

$$
\left|\int_{\T^{d}}{\rm div}(m_{n}e_{d})\varphi dx\right|=\left|\int_{\T^{d}}m_{n}e_{d}\cdot D\varphi dx\right|=\left|\int_{\T^{d}}(m_{n}e_{d}-\xi_{n})\cdot D\varphi dx\right|
$$
$$
\leq\left(\int_{\T^{d}}\frac{|m_{n}e_{d}-\xi|^{2}}{m_{n}}dx\right)^{1/2}\left(\int_{\T^{d}}|D\varphi|^{2}m_{n}dx\right)^{1/2}\rightarrow 0
$$

 On the other hand, if ${\rm div}(me_{d})=0$, then, by construction of $\mathcal F$, we have $\mathcal F(m_{n})\rightarrow\mathcal F(m)=2$. Therefore, there exists a constant $K>0$ independent of $\sigma$ such that $-\bar\lambda>K$.

We can conclude the proof choosing $\sigma$ small enough such that

\begin{align}
0\leq-\lambda\leq \sigma^{2}I<K\leq-\bar\lambda.
\end{align}
\end{proof}

\begin{prop}Under the hypothesis of Subsection \ref{sec2} the projected Mather set $\mathcal M$ does not contain any stationary calibrated curve. Moreover, if $m_{0}$ belongs to a minimal invariant set $\mathcal N$ and $m(t)$ is a calibrated curve starting from $m_{0}$ then $m(t)$ is either periodic or it has a chaotic behavior.
\end{prop}
\begin{proof} We recall that if $m(t)$ is a calibrated curve then

$$
\chi(m(t_{1}))-\chi (m(t_{2}))= \int_{t_{1}}^{t_{2}}\int_{\T^{d}} H^{*}\left(x,-\frac{w}{m}\right)dm(s)+\mathcal F(m(s))ds+\lambda (t_{2}-t_{1}).
$$

If $m(t)$ is constantly equal to $\bar m$ then we have that $\int_{\T}L\left(x,-{\bar w}/{\bar m}\right)dm+\mathcal F(\bar m)=-\bar\lambda$. As $-\lambda<-\bar\lambda$, it implies
$$
\chi(\bar m)-\chi (\bar m)= \int_{t_{1}}^{t_{2}}\int_{\T^{d}} H^{*}\left(x,-\frac{\bar w}{\bar m}\right)d\bar m(s)+\mathcal F(\bar m)ds+\lambda (t_{2}-t_{1})=(t_{2}-t_{1})(-\bar\lambda+\lambda)>0,
$$
so the contradiction.

Moreover, if $m(t)$ is a calibrated curve in $\mathcal N$, we proved in Proposition \ref{mdensity} that its trajectory has to be dense in $\mathcal N$. If the trajectory is closed then the curve is periodic and $\{m(t)\;t\in\R\}=\mathcal N$. Otherwise $\{m(t)\;t\in\R\}\subsetneq\mathcal N$ and $m(t)$ is chaotic.
\end{proof}

\section{Appendix}
We present here the result used in Proposition \ref{duplc}, which is part of an on-going work with Marco Cirant.
\begin{lemma}\label{exdualpl} Let $(\bar m,\bar w)$ be as in Proposition \ref{duplc}. For any $0\leq t_{1}<t_{2}$, the minimization problem 
$$
\bar A_{t_{1}}^{t_{2}}=\inf_{u\in \mathcal K}A_{t_{1}}^{t_{2}}(u)=\inf_{u\in \mathcal K}\left\{\int_{T^{d}}u(x,t_{2})d\bar m(t_{2})-\int_{T^{d}}u(x,t_{1})d\bar m(t_{1})\right\}
$$
where $\mathcal K$ is the set of $u\in C^{1,2}([t_{1},t_{2}]\times\T^{d})$ such that $-\partial_t u-\Delta u +H(x,Du)= F(x,\bar m)$ and $\int_{\T^{d}}u(t_{1},x)dx=0$, admits a solution.
\end{lemma}
\begin{proof} The difficulties of this minimization problem comes from the fact that, a priori, we have no regularity on the measure $\bar m$, which does not allow us to directly get the compactness of the minimizing sequence that we need. On the other hand the dynamic programming principle of $\chi$ and some local in time semiconcavity estimates help to overcome this obstacle.

 Let us recall that the dynamic programming principle for $\chi$ reads

$$
\chi(m_{0})= \inf_{(m,w)}\left( \int_{0}^{t}H^{*}\left(x,-\frac{w}{m}\right)dm(s)+\mathcal F(m(s))ds+\chi (m(t))\right)+\lambda t
$$

and that the following convex duality holds true (see Remark \ref{justif})
$$
\inf_{(m,w)\in\Pi(\bar m(t_{1}),\bar m(t_{2}))}\int_{t_{1}}^{t_{2}}\inte H^{*}\left(x,-\frac{w}{m}\right)dm(s)+\int_{\T^{d}}F(x,\bar m(s))dm(s)ds\qquad\qquad\qquad\qquad\qquad\qquad
$$
$$
\qquad\qquad\qquad\qquad\qquad\qquad\qquad\qquad\qquad=-\inf_{u\in \mathcal K}\left\{\int_{T^{d}}u(x,t_{2})d\bar m(t_{2})-\int_{T^{d}}u(x,t_{1})d\bar m(t_{1})\right\}.
$$
Fist of all we prove that for any $t_{1}<t_{2}<t_{3}$ we have that $\bar A_{t_{1}}^{t_{3}}=\bar A_{t_{1}}^{t_{2}}+\bar A_{t_{2}}^{t_{3}}$. Indeed, using the duality between the two minimization problems, we have 
$$
\bar A_{t_{1}}^{t_{2}}+\bar A_{t_{2}}^{t_{3}}=\inf_{(m,w)\in\Pi(\bar m(t_{1}),\bar m(t_{2}))}\int_{t_{1}}^{t_{2}}\inte H^{*}\left(x,-\frac{w}{m}\right)dm(s)+\int_{\T^{d}}F(x,\bar m(s))dm(s)ds
$$
$$
+\inf_{(m,w)\in\Pi(\bar m(t_{2}),\bar m(t_{3}))}\int_{t_{2}}^{t_{3}}\inte H^{*}\left(x,-\frac{w}{m}\right)dm(s)+\int_{\T^{d}}F(x,\bar m(s))dm(s)ds.
$$
If we use the dynamic programming principle of $\chi$ and the fact that $(\bar m,\bar w)$ is optimal, the expression above is equal to
$$
\int_{t_{1}}^{t_{3}}\inte H^{*}\left(x,-\frac{\bar w}{\bar m}\right)d\bar m(s)+\int_{\T^{d}}F(x,\bar m(s))d\bar m(s)ds=
$$
$$
\inf_{(m,w)\in\Pi(\bar m(t_{1}),\bar m(t_{3}))}\int_{t_{1}}^{t_{3}}\inte H^{*}\left(x,-\frac{w}{m}\right)dm(s)+\int_{\T^{d}}F(x,\bar m(s))dm(s)ds=\bar A_{t_{1}}^{t_{3}}
$$

We claim now and we prove later that, if $\{u_{n}\}_{n}$ is a minimizing sequence for $\bar A_{t_{1}}^{t_{3}}$, then $u_{n}$ uniformly convergences to a function $u\in C^{1,2}([t_{1},t_{3})\times\T^{d})$ on any $[t_{1},t]$ with $t<t_{3}$ and $u$ is admissible for $\bar A_{t_{1}}^t$. This implies that the function $u$ is a minimizer for $\bar A_{t_{1}}^{t}$ and in particular for $t=t_{2}$. If we suppose that $A_{t_{1}}^{t_{2}}(u)=\bar A_{t_{1}}^{t_{2}}+\varepsilon$, then we have
$$
A_{t_{1}}^{t_{3}}(u_{n})=A_{t_{1}}^{t_{2}}(u_{n})+A_{t_{2}}^{t_{3}}(u_{n})\geq A_{t_{1}}^{t_{2}}(u_{n})+\bar A_{t_{2}}^{t_{3}}.
$$

If we take the limit $n\rightarrow\infty$ on both side, the uniform convergence of $u_{n}$ on $[t_{1},t_{2}]$ and the fact that $u_{n}$ is a minimizing sequence for $\bar A_{t_{1}}^{t_{3}}$ give us 
$$
\bar A_{t_{1}}^{t_{3}}\geq A_{t_{1}}^{t_{2}}(u_{n})+\bar A_{t_{2}}^{t_{3}}=A_{t_{1}}^{t_{2}}(u)=\bar A_{t_{1}}^{t_{2}}+\varepsilon+\bar A_{t_{2}}^{t_{3}}=\bar A_{t_{1}}^{t_{3}}+\varepsilon
$$
which is impossible and so $u$ has to be a minimizer for $\bar A_{t_{1}}^{t_{2}}$. 

We now prove our claim and we show that the set of functions $u\in C^{1,2}([t_{1},t_{3}]\times\T^{d})$ which solves 

\be\label{semisyst}
\begin{cases}
-\partial_t u-\Laplace u +H(x,Du)= F(x,\bar m)&\mbox{in }[t_{1},t_{3}]\times\T^{d}\\
\int_{\T^{d}}u(t_{1},x)dx=0
\end{cases}
\ee
is uniformly bounded in $C^{1}([t_{1},\tau]\times\T^{d})$ for any $\tau<t_{3}$. This gives the local convergence that we used earlier. Without loss of generality we can suppose $t_{1}=0$ and $t_{3}=T$. As in Lemma \ref{2est} we argue by semiconcavity.

We consider $\xi\in\R^d$, $|\xi|\leq 1$ and we look at the equation solved by $w(t,x)=D^2u^{T}(t,x)\xi\cdot\xi$. We now define $\bar w(t,x)=w(t,x)\eta(t)$, where $\eta$ is the cutoff function $\eta(t)=(t-T)^{2}$. We choose $\xi$ such that it maximizes $\sup_{t,x}\bar w(t,x)$. 

If we derive twice in space the HJB equation in \eqref{semisyst}, then $\bar w$ solves

 $$
-\partial_t \bar w- w\eta'-\Laplace \bar w +D_{\xi\xi}H(x,Du)\eta+2D_{\xi p}H(x,Du)\cdot D^2u\xi\eta
$$
\be\label{D2eq}
+D_{pp}H(x,Du)D^2u\xi\cdot D^2u\xi\eta+D_pH(x,Du)\cdot D\bar w=D^2_{\xi\xi}F(x,m)\eta.
\ee

The cutoff function ensures the existence of a positive interior maximum of $\bar w$. At the maximum, using also the boundedness of $D_{pp}^{2}H$, the above equation implies

$$
-w\eta'-K+2D_{\xi p}H(x,Du)\cdot D^2u\xi\eta+\bar C^{-1}|D^{2}u\xi|^{2}\eta\leq D^2_{\xi\xi}F(x,m)\eta.
$$

Rearranging the terms and using the boundedness of $D^{2}_{\xi\xi}F$ we get
$$
|D^{2}u\xi|^{2}\eta\leq w\eta'+C+2C |D^{2}u\xi|\eta.
$$

As $\eta'=2\eta^{1/2}$ we can apply the Young's inequality so that $\eta'w\leq \eta/2 |D^{2}u\xi|^{2}+4$ and
$$
\frac{1}{2}|D^{2}u\xi|^{2}\eta\leq C+2C |D^{2}u\xi|\eta,
$$
which in turn gives
$$
|D^{2}u\xi|^{2}\eta\leq C.
$$
If $w^{+}$ and $\bar w^{+}$ are the positive parts of $w$ and $\bar w$, then we have our semiconcavity estimates because in $[0,\tau]\times\T^{d}$
$$
(w^{+}\eta)^{2}\leq (\bar w^{+})^{2}\leq |D^{2}u\xi|^{2}\eta^{2}\leq M.
$$

Note that $M=M(\tau)$ and it diverges when $\tau\rightarrow T$. On the other hand the estimates above, along with \eqref{uine}, gives uniform bounds on $\Vert Du\Vert_{\infty}$ on $[0,\tau]$ with $\tau<T$.

Integrating in space the HJB equation we get
$$
|\partial_{t}\int_{\T^{d}}u dx|\leq \int_{\T^{d}}|H(x,Du)|+|F(x,\bar m)|dx\leq C(\tau).
$$

As $\int_{\T^{d}}u(0,x)dx=0$, the above inequality ensures that $|\int_{\T^{d}}u(t,x)dx|\leq C(\tau)$ for any $t\leq\tau$. This gives us  ${\rm osc} (u(t,\cdot))\leq |\int_{\T^{d}}u(t,x)dx|+C \sup_{x} |D u(t,x)|\leq C(\tau)$. 

As in Lemma \ref{2est}, the boundedness of space derivatives implies also that $|\partial_{t}u(t)|\leq C(\tau)$ for any $t\leq\tau$ and so the claim.  

\end{proof}

Here we propose the proof of the existence of the smooth functions that we used in Subsection \ref{sec2}.

\begin{lemma}\label{199} For any $m_0\in {\mathcal P}(\T^d)$ and any $\ep>0$, there exists $\Phi: {\mathcal P}(\T^d)\to \R$ of class $C^1$ such that $\Phi(m_0)=1$ and $\Phi=0$ on $B_{\varepsilon}^c(m_0)$. Moreover, we can choose $\Phi$ such that 
$$
\|D_m\Phi\|_\infty \leq 10/\ep. 
$$
and with $D_{x}D_m\Phi$ bounded.
\end{lemma}

\begin{proof} Let $E$ be the compact set of $1-$Lipschitz continuous maps on $\T^d$ vanishing at $0$.
Let $(\phi_n)$ be a dense family in $E$ consisting of smooth maps. For $N$ large, we consider
$$
\Psi_N(m)=\sup_{n=1, \dots, N} \int_{\T^d} \phi_n(m-m_0). 
$$
Then  $(\Psi_N)$ is a family which is uniformly Lipschitz continuous in $\mathcal P(\T^{d})$ and converges to $\dk(\cdot, m_0)$. So, for $\eta>0$ small there exists $N$ large enough such that 
$$
\|\Psi_N-\dk(\cdot, m_0)\|_\infty\leq \eta. 
$$
Next we approximate the sup in the definition of $\Psi_N$. We consider 
$$
\Psi_N^\delta(m) = \delta \log\left( \sum_{n=1}^N  \exp\left\{ \delta^{-1}\int_{\T^d} \phi_n(m-m_0)\right\}\right).
$$
Recall that 
$$
\Psi_N(m) \leq \Psi_N^\delta(m) \leq \delta \ln(N) +\Psi_N(m).
$$
  Note that $\Psi_N^\delta$ is $C^1$, with 
\be\label{akjn}
D_m\Psi_N^\delta(m,x)= \left( \sum_{n=1}^N  \exp\left\{ \delta^{-1}\int_{\T^d} \phi_n(m-m_0)\right\}\right)^{-1} 
\sum_{n=1}^N  \exp\left\{ \delta^{-1}\int_{\T^d} \phi_n(m-m_0)\right\} D\phi_n(x).
\ee
Note that $D_m\Psi_N^\delta(m,x)$ is a convex combination of $D\phi_n(x)$, so that 
$$
| D_m\Psi_N^\delta(m,x) | \leq \sup_n | D\phi_n(x) | \leq 1.
$$
For $\delta>0$ small (depending on $N$), we have 
$$
\|\Psi_N^\delta-\dk(\cdot, m_0)\|_\infty\leq 2\eta. 
$$
In particular, for $\ep>0$, choose $\eta= \ep/5$: then 
$$
\inf_{m\in B^c_\ep(m_0)} \Psi_N^\delta(m)\geq \ep -2\eta = 3\ep/5 \qquad {\rm and} \qquad
\Psi_N^\delta(m_0) \leq 2 \eta = 2\ep/5. 
$$

Moreover, if we derive \eqref{akjn} in space we get
$$
D_{x}D_m\Psi_N^\delta(m,x)= \left( \sum_{n=1}^N  \exp\left\{ \delta^{-1}\int_{\T^d} \phi_n(m-m_0)\right\}\right)^{-1} 
\sum_{n=1}^N  \exp\left\{ \delta^{-1}\int_{\T^d} \phi_n(m-m_0)\right\} D^{2}\phi_n(x).
$$
Note that $D_{x}D_m\Psi_N^\delta(m,x)$ is a convex combination of $D^{2}\phi_n(x)$. Therefore, there exists a constant $C_{N}$ such that
$$
| D_{x}D_m\Psi_N^\delta(m,x) | \leq \sup_n | D^{2}\phi_n(x) | \leq C_{N}.
$$

To complete the result, define a map $\zeta_\ep= \R\to [0,1]$ smooth and nonincreasing, with $\zeta_\ep(s)= 0$ if $s\geq 3\ep/5$ and $\zeta_\ep(s)= 1$ for $s\geq  2\ep/5$. We can choose $\|\zeta_\ep'\|_\infty\leq 10/\ep$. The map $\Phi=\zeta_\ep\circ \Psi_N^\ep$ satisfies the claim. 
\end{proof}

\begin{lemma}\label{200} Let $A$ and $B$ be closed subsets of $\mathcal P(\T^{d})$ with an empty intersection. Then there exists a $C^1$ map $\Phi:\mathcal P(\T^{d})\to \R$ such that $\Phi=1$ on $A$, $\Phi=0$, $B$ and $D_{x}D_{m}\Phi$ bounded. 
\end{lemma}

\begin{proof} Let $\ep>0$ be the minimal distance between $A$ and $B$: 
$$
\ep:= \inf_{m\in A, \ m'\in B}\dk(m,m')>0. 
$$
Let $(m_n)$ be dense in $A$ and $\Phi_n:\mathcal P(\T^{d})\to [0,1]$ be associated with $m_n$ as in Lemma \ref{199}: $\Phi_n(m_n)=1$, $\Phi_n=0$ in $B_\ep^c(m_n)$ and $\|D_{x}D_m\Phi_n\|$ bounded. For $\delta>0$ small and $N$ large, let us set
$$
\Psi_N^\delta(m)= \delta \log\left( \sum_{n=1}^N  \exp\left\{ \delta^{-1} \Phi_n(m)\right\}\right).
$$
Note that $\Psi_N^\delta$ is $C^1$ with 
$$
D_{x}D_m\Psi_N^\delta(m,y)=  \left( \sum_{n=1}^N  \exp\left\{ \delta^{-1}\Phi_n(m)\right\}\right)^{-1} 
\sum_{n=1}^N  \exp\left\{ \delta^{-1}\Phi_n(m)\right\} D_{x}D_{m}\Phi_n(m,y).
$$

In particular, 
$$
| D_{x}D_m\Psi_N^\delta(m,x) | \leq \sup_n | D^{2}\phi_n(x) | \leq C_{N}.
$$

For $m\in B$ we have $\Phi_n(m)=0$, so that $\Psi_N^\delta(m)=\delta \ln(N)$. As $(m_n)$ is dense, we can choose, for $\eta>0$, $N$ large enough so that 
$$
\max_{m\in A} \min_{n=1, \dots, N} \dk( m,m_n)\leq \eta. 
$$
Then, for $m\in A$, there exists $n\in \{1, \dots, N\}$ with  $ \dk( m,m_n)\leq \eta$, so that (by Lipschitz continuity of $\Phi_n$)
$$
\Phi_n(m)\geq \Phi_n(m_n)-10\ep^{-1}\dk(m,m_n)\geq 1- 10\ep^{-1}\eta.
$$
Thus
$$ 
\Psi_N^\delta(m) \geq  \delta \log\left(  \exp\left\{ \delta^{-1}\Phi_n(m)\right\}\right) \geq  1- 10\ep^{-1}\eta. 
$$
We now choose $\eta>0$ such that $1- 10\ep^{-1}\eta= 2/3$ (which in turns fixes $N$), and then $\delta>0$ small such that $\delta \ln(N)\leq 1/3$. Then we have 
$$
\inf_{m\in A} \Psi_N^\delta(m)\geq 2/3\qquad {\rm and }\qquad \sup_{m\in B} \Psi_N^\delta(m)\leq 1/3. 
$$
Then conclusion follows easily. 
\end{proof}

\bibliography{KAMbib}

\providecommand{\bysame}{\leavevmode\hbox to3em{\hrulefill}\thinspace}
\providecommand{\MR}{\relax\ifhmode\unskip\space\fi MR }
\providecommand{\MRhref}[2]{%
  \href{http://www.ams.org/mathscinet-getitem?mr=#1}{#2}
}
\providecommand{\href}[2]{#2}
\begin{thebibliography}{10}

\bibitem{barles2001space}
G.~Barles and P.~E. Souganidis, \emph{Space-time periodic solutions and
  long-time behavior of solutions to quasi-linear parabolic equations}, SIAM
  Journal on Mathematical Analysis \textbf{32} (2001), no.~6, 1311--1323.

\bibitem{bogachev2009elliptic}
V.~I. Bogachev, N.~V. Krylov, and M.~R{\"o}ckner, \emph{Elliptic and parabolic
  equations for measures}, Russian Mathematical Surveys \textbf{64} (2009),
  no.~6, 973.

\bibitem{bogachev2015fokker}
V.~I. Bogachev, N.~V. Krylov, M.~R{\"o}ckner, and S.~V. Shaposhnikov,
  \emph{Fokker-planck-kolmogorov equations}, vol. 207, American Mathematical
  Soc., 2015.

\bibitem{briani2016stable}
A.~Briani and P.~Cardaliaguet, \emph{Stable solutions in potential mean field
  game systems. arxiv preprint}, 2016.

\bibitem{cannarsa2004semiconcave}
P.~Cannarsa and C.~Sinestrari, \emph{Semiconcave functions, hamilton-jacobi
  equations, and optimal control}, vol.~58, Springer Science \& Business Media,
  2004.

\bibitem{cardaliaguet2010notes}
P.~Cardaliaguet, \emph{Notes on mean field games}, Tech. report, Technical
  report, 2010.

\bibitem{cardaliaguet2013long}
\bysame, \emph{Long time average of first order mean field games and weak kam
  theory}, Dynamic Games and Applications \textbf{3} (2013), no.~4, 473--488.

\bibitem{cardaliaguet2015master}
P.~Cardaliaguet, Fran{\c{c}}ois D., J.-M. Lasry, and P.-L. Lions, \emph{The
  master equation and the convergence problem in mean field games}, arXiv
  preprint arXiv:1509.02505 (2015).

\bibitem{cardaliaguet2015second}
P.~Cardaliaguet, P.~J. Graber, A.~Porretta, and D.~Tonon, \emph{Second order
  mean field games with degenerate diffusion and local coupling}, Nonlinear
  Differential Equations and Applications NoDEA \textbf{22} (2015), no.~5,
  1287--1317.

\bibitem{cardaliaguet2013long2}
P.~Cardaliaguet, J.-M. Lasry, P.-L. Lions, and A.~Porretta, \emph{Long time
  average of mean field games with a nonlocal coupling}, SIAM Journal on
  Control and Optimization \textbf{51} (2013), no.~5, 3558--3591.

\bibitem{cardaliaguet2012long}
P.~Cardaliaguet, J.-M. Lasry, P.-L. Lions, A.~Porretta, et~al., \emph{Long time
  average of mean field games.}, NHM \textbf{7} (2012), no.~2, 279--301.

\bibitem{cardaliaguet2017long}
P.~Cardaliaguet and A.~Porretta, \emph{Long time behavior of the master
  equation in mean-field game theory}, arXiv preprint arXiv:1709.04215 (2017).

\bibitem{cirant2017existence}
M.~Cirant, \emph{On the existence of oscillating solutions in non-monotone
  mean-field games}, arXiv preprint arXiv:1711.08047 (2017).

\bibitem{cirant2018variational}
M.~Cirant and L.~Nurbekyan, \emph{The variational structure and time-periodic
  solutions for mean-field games systems}, arXiv preprint arXiv:1804.08943
  (2018).

\bibitem{de2013elements}
J.~De~Vries, \emph{Elements of topological dynamics}, vol. 257, Springer
  Science \& Business Media, 2013.

\bibitem{fathi1997solutions}
A.~Fathi, \emph{Solutions kam faibles conjugu{\'e}es et barrieres de peierls},
  Comptes Rendus de l'Acad{\'e}mie des Sciences-Series I-Mathematics
  \textbf{325} (1997), no.~6, 649--652.

\bibitem{fathi1997theoreme}
\bysame, \emph{Th{\'e}oreme kam faible et th{\'e}orie de mather sur les
  systemes lagrangiens}, Comptes Rendus de l'Acad{\'e}mie des Sciences-Series
  I-Mathematics \textbf{324} (1997), no.~9, 1043--1046.

\bibitem{fathi1998convergence}
\bysame, \emph{Sur la convergence du semi-groupe de lax-oleinik}, Comptes
  Rendus de l'Acad{\'e}mie des Sciences-Series I-Mathematics \textbf{327}
  (1998), no.~3, 267--270.

\bibitem{fathi2008weak}
\bysame, \emph{Weak kam theorem in lagrangian dynamics preliminary version
  number 10}, by CUP (2008).

\bibitem{gomes2010discrete}
D.~A. Gomes, J.~Mohr, and R.~R. Souza, \emph{Discrete time, finite state space
  mean field games}, Journal de math{\'e}matiques pures et appliqu{\'e}es
  \textbf{93} (2010), no.~3, 308--328.

\bibitem{gueant2011mean}
O.~Gu{\'e}ant, J.-M. Lasry, and P.-L. Lions, \emph{Mean field games and
  applications}, Paris-Princeton lectures on mathematical finance 2010,
  Springer, 2011, pp.~205--266.

\bibitem{huang2006large}
M.~Huang, R.~P. Malham{\'e}, P.~E. Caines, et~al., \emph{Large population
  stochastic dynamic games: closed-loop mckean-vlasov systems and the nash
  certainty equivalence principle}, Communications in Information \& Systems
  \textbf{6} (2006), no.~3, 221--252.

\bibitem{ladyzhenskaia1988linear}
O.~A. Ladyzhenskaia, V.~A. Solonnikov, and N.~N Ural'tseva, \emph{Linear and
  quasi-linear equations of parabolic type}, vol.~23, American Mathematical
  Soc., 1988.

\bibitem{lasry2006jeux}
J.-M. Lasry and P.-L. Lions, \emph{Jeux {\`a} champ moyen. i--le cas
  stationnaire}, Comptes Rendus Math{\'e}matique \textbf{343} (2006), no.~9,
  619--625.

\bibitem{lasry2006jeux2}
\bysame, \emph{Jeux {\`a} champ moyen. ii--horizon fini et contr{\^o}le
  optimal}, Comptes Rendus Math{\'e}matique \textbf{343} (2006), no.~10,
  679--684.

\bibitem{lionsmean}
P.-L. Lions, \emph{Mean-field games. cours au coll{\`e}ge de france
  (2007--2008)}.

\bibitem{lions1987homogenization}
P.-L. Lions, G.~Papanicolaou, and S.~R.~S. Varadhan, \emph{Homogenization of
  hamilton-jacobi equations}, Unpublished preprint (1987).

\bibitem{meszaros2018variational}
A.~R. M{\'e}sz{\'a}ros and F.~J. Silva, \emph{On the variational formulation of
  some stationary second-order mean field games systems}, SIAM Journal on
  Mathematical Analysis \textbf{50} (2018), no.~1, 1255--1277.

\end{thebibliography}
\bibliographystyle{amsplain}

\end{document}